%% file: MatricesArxv3.tex
\documentclass[reqno]{amsart}
  \input amssymb.sty

\input tropMacro.tex
\def\Skip{\vskip 1.5mm}
\def\pSkip{\vskip 1.5mm \noindent}

\newcommand{\etype}[1]{\renewcommand{\labelenumi}{(#1{enumi})}}
\def\eroman{\etype{\roman}}

\def\essn{{\operatorname{es}}}

\def\str{{\operatorname{str}}}

\def\({\left(}
\def\){\right)}

\def\sg{\sig}
\def\permanent{tropical determinant}
\def\multi{multicycle}

\def\nmulti{$n$-multicycle}
\def\pnmulti{proto-multicycle}

\def\invr{{\operatorname{-1}}}

\def\tang{{\operatorname{tan}}}

\newcommand{\tHinf}{\tH_{\zero}}

\def\ssingular{strictly singular}

\newcommand{\Det}[1]{ \left|{#1}\right|}

\def\tdet{determinant}

\def\trn{{\operatorname{t}}}

\def\cyc{C}

\def\pth{p}
\def\grph{G}

\def\quasi{quasi}

\def\tGz{\tG_\zero}
\def\tHz{\tH_\zero}
\def\tTz{\tT_\zero}

\def\VHmu{(V,\tHz,\mu)}
\def\RGnu{(R,\tGz,\nu)}
\def\RTGnu{(R,\tT,\tGz,\nu)}

\def\ghost{\text{ghost}}

\def\a{\alpha}
\def\sig{\sigma}

\def\one{\mathbb{1}}
\def\zero{\mathbb{0}}

\def\regular{nonsingular}

\def\um{I}
\def\zm{Z}

\def\nb{\nabla}

\def\out{{\operatorname{out}}}
\def\inn{{\operatorname{in}}}
\def\odeg{d_\out}
\def\ideg{d_\inn}

\newcommand{\len}[1]{\operatorname{\ell}(#1)}

\newcommand{\adj}[1]{\,\operatorname{adj}(#1)}
\newcommand{\trace}[1]{\operatorname{tr}(#1)}

\def\bR{\bar{R}}

\def\rone{{\one_R}}
\def\rzero{{\zero_R}}

\def\fzero{{\zero_F}}

\def\la{\lambda}

\def\R{\mathbb R}

\newcommand{\diag}{\operatorname{diag}}

\newtheorem{thm}[theorem]{Theorem}
\newtheorem*{thm*}{Theorem}
\newtheorem*{dig*}{Digression}

\newtheorem{cor}[theorem]{Corollary}

\newtheorem{lem}[theorem]{Lemma}
\newtheorem{rem}[theorem]{Remark}
\newtheorem{prop*}{Proposition}

\newtheorem{prop}[theorem]{Proposition}
\newtheorem{defn}[theorem]{Definition}
\newtheorem*{examp*}{Example}
\newtheorem*{examples*}{Examples}
\newtheorem*{remark*}{Remark}

\newtheorem*{defn*}{Definition}
\newtheorem*{note*}{Note}

\begin{document}
\title[Supertropical matrix algebra] {Supertropical matrix algebra}

\author{Zur Izhakian}\thanks{The first author is supported by the
Chateaubriand scientific post-doctorate fellowship, Ministry of
Science, French Government, 2007-2008}
\thanks{This research is supported  by the
Israel Science Foundation (grants No. 1178/06 and 448/09).}
\address{Department of Mathematics, Bar-Ilan University, Ramat-Gan 52900,
Israel} \address{ \vskip -6mm CNRS et Universit´e Denis Diderot
(Paris 7), 175, rue du Chevaleret 75013 Paris, France}
\email{zzur@math.biu.ac.il}
\author[Louis Rowen]{Louis Rowen}
\address{Department of Mathematics, Bar-Ilan University, Ramat-Gan 52900,
Israel} \email{rowen@macs.biu.ac.il}

\thanks{\textbf{Acknowledgement} The authors would like to thank the
referee for many helpful comments.}


\date{April 2008}




\begin{abstract}

 The objective of this paper is to develop a general algebraic
 theory of supertropical matrix algebra, extending \cite{zur05TropicalRank}.
 Our main results are as follows:
 \begin{itemize} \item The tropical determinant (i.e.,
 permanent) is multiplicative when all the determinants involved are
 tangible.
 \item There exists an adjoint matrix $\adj{A}$ such that the
 matrix
 $A\! \adj{A}$ behaves much like the identity matrix (times $|A|$).
\item Every matrix $A$ is a supertropical root of its
Hamilton-Cayley
 polynomial $f_A$. If these roots are distinct, then $A$ is
 conjugate (in a certain supertropical sense) to a diagonal matrix.
  \item  The tropical determinant of a matrix
 $A$ is a ghost iff the rows of $A$ are tropically dependent, iff the columns of $A$ are tropically dependent.
\item Every root of $f_A$ is a ``supertropical'' eigenvalue of
 $A$ (appropriately defined), and has a tangible supertropical eigenvector.
 \end{itemize}
\end{abstract}

\maketitle




\section{Introduction}
\numberwithin{equation}{section}

 In \cite{IzhakianRowen2007SuperTropical}, the abstract
foundations of supertropical algebra were set forth, including the
concept of a supertropical domain and supertropical semifield. The
motivation was to overcome the difficulties inherent in studying
polynomials over the max-plus algebra, by providing an algebraic
structure that encompasses the max-plus algebra, thereby
permitting a thorough study of polynomials and their roots and a
direct algebraic-geometric development of tropical geometry.

Similarly, although there has been considerable interest recently
in linear algebra over the max-plus algebra
\cite{ABG,Sturm4,Sturm3}, the weakness of the inherent structure
of the max-plus algebra has hampered a systematic development of
the matrix theory. The object of this paper is to lay the
groundwork for such a theory over a supertropical domain, which
yields analogs  of much the classical matrix theory for the
max-plus algebra and also explains why other parts do not carry
over.

The max-plus algebra is a special kind of idempotent semiring. In
general, the matrix semiring over a semiring is also a semiring
(to be described below in detail), but often loses some of its
properties. So we also need to pinpoint some of those properties
that are preserved in such matrix semirings. Our underlying
structure is a
 \textbf{semiring with ghosts}, which we recall from \cite{IzhakianRowen2007SuperTropical} is a
triplet $\RGnu,$ where $R$ is a semiring with a unit element
$\rone$ and with zero element $\rzero$ (satisfying $\rzero \, r =
r \, \rzero   = \rzero$ for every $r\in R$, and often identified
in the examples with $-\infty$, as indicated below), $\tGz = \tG
\cup \{ \rzero \} $ is a semiring ideal called the \textbf{ghost
ideal}, and $\nu : R  \to  \tGz,$ called the \textbf{ghost map},
is an idempotent semiring homomorphism (i.e., which preserves
multiplication as well as addition).

We write $a^{\nu }$ for $\nu(a)$, called the $\nu$-\textbf{value}
of $a$.  Thus, $\rone^\nu $  is multiplicatively idempotent, and
serves as the unit element of $\tGz$. Two elements $a$ and $b$ in
$R$ are said to be \textbf{matched} if they have the same
$\nu$-value; we say that $a$ \textbf{dominates} ~$b$ if $a^\nu \ge
b^\nu$.

For tropical applications, we focus on the \textbf{tangible
elements}, which in this paper are defined as $\tT = R\setminus
\tGz$; they are defined more generally in
\cite{IzhakianRowen2007SuperTropical} (cf.~Remark~\ref{reduc}
below).  We write $\tTz$  for $\tT \cup \{ \zero_R \}$. (Although
$ \zero_R $ is a ghost element, being part of the ghost ideal, it
is useful to consider it together with the tangible elements when
considering linear combinations.)

 Next, in  \cite[Definition
3.5]{IzhakianRowen2007SuperTropical}, we defined a
\textbf{supertropical semiring}, which is a commutative semiring
with ghosts satisfying the extra properties: \pSkip
\begin{itemize}
 \item  $a+b   =  a^{\nu } \quad \text{if}\quad a^{\nu } =
 b^{\nu}$; \pSkip
 \item $a+b  \in \{a,b\},\ \forall a,b \in R \;s.t. \; a^{\nu }
\ne b^{\nu }.$ (Equivalently, $\tGz$ is ordered, via $a^\nu \le
b^\nu$ iff $a^\nu + b^\nu = b^\nu$.) \pSkip

\end{itemize}
Thus, $a^\nu = \rzero^\nu$ iff $$a = a+\rzero = \rzero^\nu =
 \rzero.$$ It follows that $a+b = \rzero$ iff $\max \{a^\nu, b^\nu\} = \rzero,$ iff $a= b = \rzero.$
 Hence, no nonzero element has an additive inverse.

In studying supertropical semirings in
\cite{IzhakianRowen2007SuperTropical}, we defined
 a \textbf{supertropical domain} to be a supertropical semiring for which $\tT$ is a monoid;
 we also assume here that the map $\nu _\tT : \tT \to \tG$ (defined as the restriction from $\nu$ to $\tT$) is
 onto. (See \cite[Remark~3.11]{IzhakianRowen2007SuperTropical} for
 some immediate consequences of this definition, including a version of cancellation.)
 We also defined a \textbf{supertropical semifield}
to be a supertropical domain $\RGnu$ for which $\tT$ is a
group.\Skip

Whereas the paper ~\cite{IzhakianRowen2007SuperTropical} focused
on the theory of polynomials and their roots over supertropical
semifields, in this paper
 we turn to the matrix theory of semirings with ghost ideals, and so bring in tropical determinants,
  i.e., permanents,   and tropical linear algebra. We obtain a multiplicative rule for the
tropical determinant (Theorem \ref{thm:detOfProd}), a tropical
theory of the adjoint matrix (Corollary   \ref{thm:idm1}, and
Theorems \ref{thm:idm0} and \ref{thm:idm}), a version of the
Hamilton-Cayley theorem (Theorem \ref{hamilton-Cayley}),
supertropical eigenvalues (Theorem \ref{eigen8}), and the fact
that a matrix is singular iff its rows, or its columns, are
(tropically) dependent (Theorem \ref{thm:base}). Some of our
results follow \cite{zur05TropicalRank}, which handled the special
case where $R$ is the ``extended tropical semiring'' of the real
numbers; the proofs here are somewhat more conceptual. Theorem
\ref{thm:base} is extended in
\cite{IzhakianRowen2009TropicalRank}, which relies on this paper.

``Linear algebra over a semiring'' is the title of a chapter in
Golan's book \cite[Chapter 17]{golan92}, and there already exists
a sizeable literature concerning linear algebra for the max-plus
algebra, as summarized in~\cite{ABG}; one major result there is
the existence of eigenvectors for matrices over max-plus algebras.
(See also \cite{Sturm3} and \cite{Sturm4} for results concerning
tropical determinants and tropical rank.) Nevertheless, the ghost
ideal here changes the flavor considerably, enabling us to define
and utilize adjoint matrices and also obtain a supertropical
version of the Hamilton-Cayley theorem, together with applications
to obtain tangible supertropical eigenvectors for all roots of the
characteristic polynomial.

To clarify our exposition for those versed in tropical
mathematics, the examples in this paper are presented for the
extended tropical semiring \cite{zur05TropicalAlgebra}, the
motivating example for supertropical semirings. For this semiring,
 denoted as $D(\Real)$, we have $\tT = \tG = \Real$,
$\rone = 0$, and $\rzero = \tUniS$, where its operations, $\TrS$
and ~$\cdot$ , are respective modifications of the standard $\max,
\ + \ $ operations over the reals. In other words, we use
\textbf{logarithmic notation} in all of our illustrations, whereas
in the theorems, we use multiplicative notation which is more in
accordance with the ring-theoretic structure, our source of
intuition. We hope this does not cause undue confusion.

Throughout this paper, as in \cite
{IzhakianRowen2007SuperTropical}, we assume that $\nu$ is given by
\begin{equation}\label{supert} \nu (a) = a+a.\end{equation} Although there are more general situations
of interest in supertropical algebra, they can often be reduced to
the setting here because of the following observation:

 \begin{rem}\label{reduc} Suppose $\RGnu,$
 is a semiring with ghosts satisfying $(\rone+\rone)^\nu =\rone^\nu $, (but not necessarily satisfying $ \rone + \rone=\rone^\nu $),
and $\tT \subseteq R \setminus \tGz$ is
 any multiplicative monoid. Taking $\tGz' = R
\rone^\nu \subseteq \tGz,$
 one can define a new semiring structure $\bar R = \tT \cup \tGz' $, as follows:

  Multiplication is the restriction to $\bR$ of multiplication in $R $, so
 $\rzero $ remains the zero  element, $\rone$~remains the unit element,
 and  $\rone^\nu $ still is multiplicatively idempotent.

  The new addition in $\bR$ is given by $\rzero  +r = r = r+ \rzero $ for all
  $r\in R;$ but now, the sum of two
elements $a$ and $b$ in $\bR$ is defined to be their sum in $ R$
if it lies in $\tT$, and is $(a+b)^\nu $ otherwise. In particular,
$a+a= a^\nu $ in $\bar R$, and $\bR$ is a supertropical semiring.
The ghost ideal of $\bR$ is $R \rone^\nu$, and the tangible part
is ~$\tT.$

Thus, we see that the ``tangible'' part of the algebraic
structures of $R$ and $\bR$ are the same, and in particular the
theorems in this paper about $M_n(\bR)$ also hold for $M_n(R)$.
  \end{rem}

  We write ``$a = b + \text{ghost}$'' to indicate that $a$
equals $b$ plus some undetermined ghost element. This can happen
in two ways: Either $a \in \tT$ (in which case $a=b$), or $a\in
\tG$ with $a^\nu\ge b^\nu$ (in which case $a = b+a$).

\begin{rem} If  $a = b + \text{ghost}$, then Equation~ \eqref{supert} implies $a + b = b +b + \text{ghost}\in \tGz$, although the converse might fail. \end{rem}

One difference with \cite{IzhakianRowen2007SuperTropical} is that
here we do not require our semirings to be commutative, since we
must deal with semirings of matrices. Nevertheless, we do have the
following important property:

\begin{rem}[\textbf{The Frobenius property}]\label{Frob0}  $(r+z)^m$ equals $r^m
+ z^m + \text{ghost}$, for all $m \in \mathbb N ^+$, $r\in R$, and
central $z\in R$. This is because $$(r+z)^m = r^m + z^m + \sum _{1
\le i <m} \binom mi r^i z^{m-i},$$ and each summand in the
summation is ghost since $\binom mi
> 1$ for $1 \le i <m$.
It follows that $(r+z)^m+ (r^m + z^m)$ is ghost, whenever $z$ is
central.
\end{rem}

\section{Tropical modules and matrices}

Modules over semirings (called \textbf{semimodules} in
\cite{golan92}) are defined just as modules over rings, except
that now the additive structure is that of a semigroup instead of
a group. (Note that subtraction does not enter into the other
axioms of a module over a ring.) Let us state this explicitly, for
the reader's convenience.

\begin{defn}\label{def:module0} Suppose $R$ is a semiring.
An \textbf{$R$-module} $V$ is a semigroup $(V,+,\zero_V)$ together
with a scalar multiplication $R\times V \to V$ satisfying the
following properties for all $r_i \in R$ and $v,w \in V$:
\begin{enumerate}
    \item $r(v + w) = rv + rw;$ \pSkip
    \item $(r_1+r_2)v = r_1v + r_2 v;$ \pSkip
    \item $(r_1r_2)v = r_1 (r_2 v);$ \pSkip
    \item $\rone v =v;$ \pSkip
    \item $\rzero v =\zero_V = r \zero_V .$
  \end{enumerate}
 \end{defn}

 Note that this definition of module over a
semiring $R$ coincides with the usual definition of module when
$R$ is a ring, taking $-v = (-\rone)v  .$

\begin{defn}\label{def:module} Suppose $\RGnu$ is a semiring with ghosts.
An \textbf{$R$-module with ghosts} $\VHmu$  is an $R$-module $V$,
together with a \textbf{ghost submodule} $\tHinf$ and an
$R$-module projection
$$\mu : \ V \To \tHinf$$ satisfying the following axioms for all
$r\in R$ and $v,w\in V:$
\begin{enumerate}
    \item $\mu(rv) =  r \mu (v)= r^\nu v$; \pSkip
    \item $\mu(v+w) = \mu(v) + \mu(w)$. \pSkip
 \end{enumerate}
 \end{defn}
\noindent  Note that (1) implies $\tGz V \subseteq \tHinf.$

Rather than developing the general module theory here, we content
ourselves with the following example.

\begin{example}  The direct sum $V = \bigoplus_{j \in \tJ} R$ of  copies
    (indexed by $\tJ$) of a  supertropical semiring $R$  is denoted as~$R^{(\tJ)},$
    with zero element $\zero_V = ( \rzero).$
    The ghost submodule is $\tGz^{(\tJ)}$.
    When $R$ is a supertropical semifield, $R^{(\tJ)}$~is called a
\textbf{tropical vector space} over $R$.\pSkip

If we take $\tJ = \{ 1, \dots, n \}$, then the tropical module
$R^{(\tJ)}$ is denoted as $R^{(n)}$, which is  the main example of
tropical linear algebra. The \textbf{tangible vectors} of
$R^{(n)}$ are defined as those $(a_1, \dots, a_n)$ such that each~
$a_i\in \tTz$, but with some $a_i \ne \rzero.$ (Note that there
may be vectors that are neither tangible nor ghost, having some
tangible components and some ghost components.)
\end{example}

\begin{defn} The \textbf{standard base} of $R^{(n)}$ is defined as
$$e_1 = (\rone,\rzero, \dots, \rzero),
\quad e_2 = (\rzero,\rone,\rzero, \dots, \rzero), \quad \dots,
\quad e_n = (\rzero,\rzero, \dots, \rone).$$
\end{defn}

Note that every element $(r_1, \dots, r_n)$ of $R^{(n)}$ can be
written (uniquely) in the form $\sum_{i=1}^n r_i e_i$.

\subsection{Matrices over semirings with ghosts}

It is standard that for any semiring $R$, we have the semiring
$M_n(R)$ of $n \times n$ matrices with entries in ~$R$, where
addition and multiplication are induced from~$R$ as in the
familiar ring-theoretic matrix construction. The unit element
 of
 $M_n(R)$ is the \textbf{identity matrix} $\um$ with $\rone$ on the main diagonal and
whose off-diagonal entries are $\rzero$.

Given the designated ghost ideal $\tGz$ of  $R = \RGnu$, we define
the \textbf{ghost ideal} $M_n(\tGz)$ of $M_n(R)$ and thus we
obtain the \textbf{matrix semiring with ghosts} $(M_n(R),
M_{n}(\tGz), \nu_*)$, where the ghost map  $\nu_*$ on $M_n(R)$ is
obtained by applying $\nu$ to each matrix entry.

\begin{rem}\label{Frob1} The Frobenius property (Remark \ref{Frob0}) implies that for any matrix $A$
over a commutative semiring $R$ with ghosts and any $\a \in R,$
the matrix $(A+\a I)^m$ equals $A^m + \a^m I + \text{ghost}$, in
$M_{n}(R).$ Note that $(A+\a I)^m$ can differ from $A^m + \a^m I;$
for example, take
 $A = \vMat{\rzero}{\rone}{\rone}{\rzero}$ with $m = 2$.
\end{rem}

\section {Tropical \tdet s}\label{vectorperm}

For the remainder of this paper, unless otherwise  specified, we
only consider matrices over supertropical domains $R = \RGnu$. A
typical matrix is denoted as $A =(a_{i,j})$; for example, the zero
matrix is $(\rzero)$. The tropical version of the determinant must
be the permanent, since we do not have negation at our disposal.
Nevertheless, its function in supertropical algebra is the analog
of the familiar determinant. In \cite{zur05TropicalAlgebra}, a
counterexample was given to the proposed formula $\Det{AB} =
\Det{A}\, \Det{B}.$  Let us see why such counterexamples exist, by
providing a conceptual development of the tropical determinant
that indicates when the formula does hold. As in classical
algebra, when we study tropical determinants, we assume as a
matter of course that the base semiring $R$ is commutative.

\begin{thm}\label{formula1}  Suppose $V = R^{(n)}$, taken with the standard
base $(e_1, \dots, e_n)$, over a supertropical (commutative)
semiring $R = \RTGnu$.

Define the function $\Phi_\gm: V^{(n)} \to R$ by the following
formula, where $v_i = (v_{i,1}, \dots, v_{i,n})$:
\begin{equation}\label{det1}
\Phi_\gm (v_1, \dots, v_n) =  \gm \sum _{\pi \in S_n}v_{1,\pi (1)}
\cdots v_{n,\pi (n)},\end{equation} where $\gm  \in R$ is fixed.
Then $\Phi_\gm $ satisfies the following properties: \pSkip
\begin{enumerate}\item $\Phi_\gm $ is linear in each component, in the
sense that $$ \begin{array}{ll}
                \Phi_\gm (v_1, \dots, \a _iv_i+\a _i'v_i',
\dots, v_n)= &
              \a _i\Phi_\gm (v_1, \dots, v_i, \dots, v_n)+ \a
_i'\Phi_\gm (v_1, \dots, v_i', \dots, v_n),
              \end{array}
  $$ for all $\a _i ,
\a _i'\in R$ and $v_i,v_i'$ in $V$. \pSkip

 \item  $\Phi_\gm (v_1,
\dots, v_n) \in \tGz$ if $v_i = v_j$ for some $i\ne j$. \pSkip

 \item  $\Phi_\gm (v_1,
\dots, v_n) = \rzero$ if $v_i = \zero_V$ for some $i$. \pSkip

\item $\Phi_\gm (v_{\pi(1)}, \dots, v_{\pi(n)}) = \Phi_\gm (v_1,
\dots, v_n),$ for all $\pi \in S_n$. \pSkip

\item $\Phi_\gm (e_1, \dots, e_n) = \gm .$ \pSkip
\end{enumerate}
Furthermore, $\Phi_\gm $ is unique up to ghosts, in the sense that
if $\Phi_\gm '$ is another function satisfying the same properties
(1)--(5), then either
$$\Phi_\gm '(v_1, \dots, v_n) = \Phi_\gm (v_1, \dots, v_n)$$
or $\Phi_\gm '(v_1, \dots, v_n)\in \tGz$, with $(\Phi_\gm '(v_1,
\dots, v_n))^\nu \ge (\Phi_\gm (v_1, \dots, v_n))^\nu$.
\end{thm}
\begin{proof}
First  of all, note that Formula \Ref{det1} satisfies the
conditions (1)--(5) of the assertion. Conversely, suppose
$\Phi'_\gm $ satisfies these conditions. Since $v_i = \sum v_{i,j}
e_j,$ we have (by linearity)
$$\Phi'_\gm (v_1, \dots, v_n) = \sum _{j_1, \dots, j_n} v_{1,j_1}\cdots
v_{n,j_n}\Phi'_\gm (e_{j_1},\dots  ,e_{j_n}) .$$
 When any $j_s =
j_t$, we get $\Phi'_\gm (e_{j_1}, \dots, e_{j_n}) \in \tGz$ by
property (2). If such ghost terms do  not dominate all the
$v_{1,\pi(1)}\cdots v_{n,\pi(n)}\Phi'_\gm (e_{\pi(1)},\dots
,e_{\pi(n)}),$ $\pi \in S_n$, then
$$\Phi'_\gm (v_1, \dots, v_n) = \sum _{\pi \in S_n}
v_{1,\pi(1)}\cdots v_{n,\pi(n)}\Phi'_\gm (e_{\pi(1)},\dots
,e_{\pi(n)}) =\gm\sum _{\pi \in S_n} v_{1,\pi(1)}\cdots
v_{n,\pi(n)}$$ since, by conditions (4) and (5),
$$\Phi'_\gm (e_{\pi(1)},\dots
,e_{\pi(n)})=\Phi'_\gm (e_1, \dots, e_n) = \gm .$$ This proves the
last assertion.
\end{proof}

\begin{rem}\label{detrep0} Condition (1) implies condition (3). Indeed,
$$\Phi_\gm (v_1, \dots, \zero_V, \dots, v_n) = \Phi_\gm (v_1, \dots, \zero_R v_i, \dots, v_n)
= \zero_R\Phi_\gm (v_1, \dots, v_i, \dots, v_n) = \zero_R.$$
\end{rem}

\begin{rem}\label{detrep} Actually, the same proof shows that $\Phi_\gamma$ satisfies the
following stronger property than (2):

 \begin{itemize} \item $\Phi_\gm (v_1,
\dots, v_n) \in \tGz$ if $v_i^\nu = v_j^\nu$ for some $i\ne j$ (in
other words, if the corresponding components have the same
$\nu$-values).
 \end{itemize}
Conversely, (1) and (4) imply that it is enough to verify (2) for
 the standard base $e_1, \dots, e_n.$ \end{rem}

 When $\gm  = \rone,$ we
denote $\Phi_\gm (v_1, \dots, v_n)$ as $\Det{v_1, \dots, v_n}$ and
call this the \textbf{normalized} version of~Formula~\Ref{det1}.
On the other hand, Theorem \ref{formula1} points to a strange
phenomenon: Ghosts produce ``noise'' which disrupts attempts to
provide an analog to the classical determinantal theory, as we
shall see.

We define the \textbf{tropical determinant} of a matrix $A =
(a_{i,j})$ as in Formula \Ref{det1} (normalized) applied to the
rows of $A$:
\begin{equation}\label{det2}
\Det{(a_{i,j})} = \sum _{\pi \in S_n}a_{1,\pi (1)} \cdots a_{n,\pi
(n)},\end{equation}
 which is the formula given in~\cite{zur05TropicalRank}. (Also see
 Remark~\ref{transpose}.)

\begin{rem}\label{transpose} Defining the \textbf{transpose}
$(a_{i,j})^{\trn}$ to be $(a_{j,i}),$ we have
$$\Det{(a_{i,j})^{\trn} } = \Det{(a_{i,j})},$$ in view of
Theorem \ref{formula1}, since
  $$\sum _{\pi \in S_n}a_{1,\pi (1)} \cdots
a_{n,\pi (n)} = \sum _{\pi \in S_n}a_{\pi (1),1} \cdots a_{\pi
(n),n}.$$\end{rem}

 As in classical linear algebra, we thus have   analogous results
 if we use columns instead of rows.

\begin{thm}\label{thm:detOfProd} For any $n \times n$ matrices over a supertropical semiring $R$, we have
$$\Det{AB}^\nu \ge \Det{A}^\nu \Det{B}^\nu,$$ with $\Det{AB} = \Det{A}\, \Det{B}$ whenever $\Det{AB} $ is
tangible. (In other words, $\Det{AB} = \Det{A}\Det{B}+\ghost.$)
\end{thm}
\begin{proof} Define $\Phi_{\Det{B}}(A) =  \Det{AB}.$ This
satisfies all of the properties of Theorem \ref{formula1}, taking
$\gm = \Det{B},$ so must be $\gm \Det{A} = \Det{A}\,\Det{B}$
except   when $\Det{AB}$ is ghost and dominates $\Det{A} \,
\Det{B}$.
\end{proof}

\subsection{Tropically singular and nonsingular matrices}

We start this subsection with the supertropical version of the
terms ``nonsingular'' and ``singular, '' to be contrasted with the
classical notion of invertibility:
\begin{defn}
 A matrix $A$ is \textbf{\regular } if $\Det{ A }\in \tT$;
on the other hand, when $\Det{ A } \in \tGz$,  we say that $A$ is~
\textbf{singular}.
 When $\Det{ A } = \rzero$, we say that $A$ is
\textbf{\ssingular}.
\end{defn}

Note that if $\Det{A}$ is any ghost $\ne \rzero,$ then $A$ is
singular but not \ssingular. Although the two concepts of singular
and \ssingular\ are analogous, the approach to their theories are
quite different.

\begin{remark}\label{rmk:zeroDet} Let us study
determinants via permutations, utilizing Formula \eqref{det2} to
analyze~$\Det{A}$ where $A = (a_{i,j})$. Clearly $$\nu(\Det{A}) =
\nu(a_{1,\sig(1)} \;  \cdots \; a_{n,\sig(n)})  $$ iff
$a_{1,\sig(1)} \;  \cdots \; a_{n,\sig(n)}$, $\sig \in S_n$,  has
the maximal $\nu$-value of all such products.  We say a
permutation $\sig \in S_n$
 \textbf{attains}~$\Det{A}$ if $\Det{A}^\nu = (a_{\sig (1),1} \cdots a_{\sig
(n),n})^\nu.$

\begin{itemize}
\item By definition, some permutation always attains $\Det{A}$.
\pSkip \item If  there is a unique permutation $\sig$ which
attains $\Det{A}$, then  $\Det{A}= a_{1,\sig(1)} \;  \cdots \;
a_{n,\sig(n)}$. In this case, when $\Det{A}$  is ghost, then some
$a_{i,\sig(i)}$ must be ghost. \pSkip

\item If at least two permutations  attain $\Det{A}$, then $A$
must be singular. Note in this case that if we replaced all
nonzero entries of $A$ by tangible entries of the same
$\nu$-value, then $A$ would still be singular. \pSkip

\item When $A$  is nonsingular, there is a unique permutation
$\sig$  which attains $\Det{A}$; in this case each $a_{i,\sig(i)}$
is tangible.  \pSkip

 \item  When $\Det{ A } =
\rzero$, then every permutation attains $\Det{A}$, so we must have
$$ a_{1,\sig(1)}
\;  \cdots \; a_{n,\sig(n)} = \rzero $$ for each $\sig \in S_n$.
Accordingly,  for each permutation $\sig$, at least one of the
$a_{i,\sig(i)}$ is $\rzero$ (where $i$ depends on $\sig$).
\end{itemize}
\end{remark}

 Thus, $\Det{ A } = \rzero$ iff ``enough'' entries are $\rzero$ to force each
summand in Formula \eqref{det2} to be $\rzero$.  This is a very
strong property, which in classical matrix theory provides a
description of singular subspaces. We elaborate this idea in
Proposition~\ref{ssing}.

 We write $P_\sig$ for the
\textbf{permutation matrix} whose entry in the $(i, \sig (i))$
position is $\rone$ (for each $1\le i \le n$) and $\rzero$
elsewhere; $P_\sig$ is \regular \ for any $\sig \in S_n$.
Likewise, we write $\diag\{a_1, \dots, a_n\}$ for the
\textbf{diagonal} matrix whose entry in the $(i,i)$ position is
$a_i \in R$ and $\rzero$ elsewhere.

\begin{example} Any permutation matrix ${P_\sig}$
is (classically) invertible; indeed, ${P_\sig}^{\invr}
=P_{\sig^{\invr}}$. Also, the diagonal matrix $\diag\{a_1,\dots,
a_n\}$ is invertible iff each $a_i$ is invertible in $R$, for then
$$\diag\{a_1, \dots, a_n\}^\invr = \diag\{a_1^\invr, \dots,
a_n^\invr\}.$$
\end{example}

The following easy result should be well known.
\begin{prop}\label{prop:invr} Suppose $R$ is a supertropical semiring. A matrix $A\in M_n(R)$ is (multiplicatively)
invertible, iff $A$ is a product of a permutation matrix with an
 invertible diagonal matrix.
\end{prop}

\begin{proof} Any invertible  matrix $A$
is \regular , by Theorem~\ref{thm:detOfProd}, since $\Det{A
A^\invr} = \rone.$ Thus, for the permutation~$\sig$ attaining
$\Det{A}$, we have $\{a_{\sig (1),1 }, \ \dots, \ a_{\sig
(n),n}\}\in \tT.$. Replacing $A$ by $ P_{\sig ^\invr}A,$ we may
assume that the diagonal of $A$ is tangible; then, multiplying
through by a suitable diagonal matrix, we may assume that the
diagonal of $A$ is the identity matrix $\um$. In other words, $A$
has the form $A = \um + B$ for some matrix $B$ which is $\rzero$
on the diagonal. Also, write $A^\invr = D' + B'$ where $D'$ is
diagonal and $B'$ is $\rzero$ on the diagonal. But then, $\um =
AA^\invr = D' + BD' + B' +BB',$ which can be $\rzero$ off the
diagonal only if $B = B' = (\rzero).$
\end{proof}


\begin{remark} \label{rmk:maximalSG}
The set $$\tW = \{ \ Q_\sig = P_\sig D \ |  \  D \text{ is
  invertible diagonal}\},$$ which by Proposition \ref{prop:invr}
comprises the unique maximal subgroup of $M_n(R)$ (having the same
identity element $\um$),
 is in fact the (affine) Weyl group when $\tT = \Int$;
 cf.~\cite{Hall}.
\end{remark}

Thus, invertibility in supertropical matrices is a strong concept,
and we want to consider the weaker notion of \regular ity. We
start by asking when the power of a nonsingular matrix is
nonsingular.

\begin{example}\label{2by2} Let us compute $\Det{A^2}$, for any $2\times
2$ matrix $$A = \( \begin{matrix} a_{1,1} & a_{1,2} \\
a_{2,1} & a_{2,2}
\end{matrix}\) ,$$ and compare it to $\Det{A}$.
Clearly $A^2 = \( \begin{matrix} (a_{1,1})^2 + a_{1,2} \, a_{2,1} & a_{1,2} \, (a_{1,1}+a_{2,2}) \\
a_{2,1} \, (a_{1,1}+a_{2,2}) & (a_{2,2})^2 + a_{1,2} \, a_{2,1}
\end{matrix}\), $
so \begin{equation}\begin{aligned}\Det{A^2} & = ((a_{1,1})^2 +
a_{1,2} \, a_{2,1})((a_{2,2})^2 + a_{1,2} \, a_{2,1}) +
(a_{1,1}+a_{2,2})^2 \ a_{1,2} \, a_{2,1} \\ &= 
\nu((a_{1,1})^2+(a_{2,2})^2 \, a_{1,2} \, a_{2,1})+(a_{1,1})^2 \,
(a_{2,2})^2+(a_{1,2})^2 \, (a_{2,1})^2
 + \nu(a_{1,1} \, a_{2,2} \, a_{1,2} \, a_{2,1}) \\ &=
 \nu(((a_{1,1})^2+(a_{2,2})^2) \, a_{1,2} \, a_{2,1})
 +(a_{1,1} \, a_{2,2}+a_{1,2} \, a_{2,1})^2.\end{aligned}\end{equation}
The right side is ghost when
\begin{equation}\label{except}\nu(((a_{1,1})^2+(a_{2,2})^2) \, a_{1,2} \, a_{2,1}) \ge  \nu((a_{1,1}
\, a_{2,2}+a_{1,2} \, a_{2,1})^2).\end{equation} Assuming that
$a_{1,1}^\nu \ge a_{2,2}^\nu$, we get \eqref{except} iff $\nu
((a_{1,1})^2) \ge  \nu(a_{1,2} \, a_{2,1}) \ge  \nu
((a_{2,2})^2)$. (The situation for $a_{1,1}^\nu \le a_{2,2}^\nu$
is symmetric.) Let us examine the various cases in turn, where
$A^2$ is \regular.

\begin{description}
    \item[Case I] $\nu((a_{1,1})^2) = \nu((a_{2,2})^2 ) > \nu(a_{1,2} \, a_{2,1}).$ Then
$$A^2 = \( \begin{matrix} (a_{1,1})^2   & a_{1,2} \, a_{1,1}^\nu  \\
a_{2,1} \, a_{1,1}^\nu & (a_{2,2})^2,
\end{matrix}\) ,$$
so the entries of $(a_{1,1}\um)A$ and $A^2$ are $\nu$-matched, and
  we see by iteration that $A^{2^u}$ is \regular \ for every
$u$, and thus every power of $A$ is \regular . \pSkip
    \item[Case II] $\nu((a_{2,2})^2) \le \nu((a_{1,1}) ^2) < \nu(a_{1,2} \, a_{2,1}).$ Then
$$A^2 = \( \begin{matrix} a_{1,2} \, a_{2,1}   & a_{1,2} \, a_{1,1}  \\
a_{2,1} \, a_{1,1}  & a_{1,2} \, a_{2,1}
\end{matrix}\) ,$$
(where the off-diagonal terms are made ghost if  $a_{1,1}^\nu =
 a_{2,2}^\nu $), which has the form of Case I; hence, every power of
$A^2$, and thus of $A$, is \regular . \pSkip

    \item[Case III] $\nu((a_{1,1})^2) > \nu((a_{2,2}) ^2) >
\nu(a_{1,2} \, a_{2,1}).$ Then
$$A^2 = \( \begin{matrix} (a_{1,1})^2   & a_{1,2} \, a_{1,1}  \\
a_{2,1} \, a_{1,1}  & (a_{2,2})^2
\end{matrix}\)  = (a_{1,1}\um)A',$$ where $A'$ differs
from $A$ only in the $(2,2)$-entry, whose $\nu$-value has been
reduced by a factor of $\frac{a_{2,2}}{a_{1,1}}.$ Taking a high
enough power of $A$ will reduce $(a_{2,2})^2$ until it is
dominated by 
$a_{1,2} \, a_{2,1}$, and thus yield a singular
matrix. Thus, some power of $A$ will always be singular, even
though $A^2$ need not be singular.
\end{description}
Summarizing, $A^2$ \regular \ implies every power of $A$ is
\regular\ except in Case III, which  for any~$k$ provides an
example where $A^k$ is \regular \ but $A^{k+1}$ is singular.
\end{example}

\subsection{The digraph of a supertropical matrix}

One major computational tool in tropical matrix theory is the
\textbf{weighted digraph} $\grph = (\mathcal V, \mathcal E)$ of an
$n\times n$ matrix $A = (a_{i,j})$, which is defined to have
vertex set $\mathcal V =\{ 1, \dots, n\}$, and an edge $(i,j)$
from $i$ to $j$ (of \textbf{weight} $a_{i,j}$) whenever $a_{i,j}
\ne \rzero$.

We use \cite{Gibbons85} as a general reference for graphs. We
always assume that $\mathcal V = \{ 1, \dots, n \}$, for
convenience of notation. The \textbf{out-degree}, $\odeg(i)$, of a
vertex $i$  is the number of edges emanating from $i$, and the
\textbf{in-degree}, $\ideg(j)$, is the number edges terminating at
$j$.  A \textbf{sink} is a vertex $j$ with $\odeg(j) = 0$, while a
\textbf{source} is a vertex $j$ with $\ideg(j) = 0$.

The \textbf{length} $\len{\pth}$ of a path $\pth$ is the number of
edges of the path. A path is \textbf{simple} if each vertex
appears only once. A \textbf{simple cycle} is a simple path for
which $\odeg(i) = \ideg(i) =1$ for every vertex $i$ of the path;
thus, the initial and terminal vertices are the same.  A simple
cycle of length $1$ is then a \textbf{loop}. A simple cycle
repeated several times is called a \textbf{cycle}; thus, for some
$m$, $\odeg(i) = \ideg(i) = m$ for every vertex $i$ of the cycle.

It turns out that the only edges of use to us are those that are
parts of cycles. Accordingly, we define the \textbf{reduced
digraph} $\grph _A$  of $A$ to be the graph obtained from the
weighted digraph by erasing all edges that are not parts of
cycles. Consequently, if there is a path from $i$ to $j$ in $\grph
_A$, there also is a path from $j$ to $i$. Hence, $\grph _A$ can
be written as a disjoint union of connected components, in each of
which there is a path between any two vertices.

The \textbf{weight} $w(\pth)$ of a path $\pth$ is defined to be
the tropical product of the weights of the edges comprising
$\pth$, counting multiplicity. The \textbf{average weight} of the
path $\pth$ is $\root \ell \of {w(\pth)}$,  where $\ell = \ell(p)$
is the \textbf{length} of the path, i.e., the number of edges in
the path. (As always, our product, being tropical, is really the
sum, so we indeed are taking the average.) We order the weights
according to their $\nu$-values. Then the $(i,j)$-entry of $A^k$,
where $A$ is a tangible matrix, corresponds to the highest weight
of all the paths of length $k$ from $i$ to $j$, and is a ghost
whenever two distinct paths of length $k$ have the same highest
weight.

We define a $k$-\textbf{multicycle} $\cyc$ in a digraph to be the
union of disjoint simple cycles, the sum of whose lengths is $k$;
its \textbf{weight} $w(C)$ is the product of the weights of the
component cycles. Thus, each $n$-multicycle passes through all the
vertices; $n$-multicycles are also known in the literature as
\textbf{cyclic covers}, or \textbf{saturated matchings}.

\begin{rem}\label{matchdet1}
Writing a permutation $\sig$ as a product $\sig _1\cdots \sig_t$
of disjoint cyclic permutations, we see that each permutation
 corresponds to an \nmulti.
Conversely, any \nmulti\ corresponds to a permutation, and their
highest weight in $\grph _A$ matches~$\Det{A}$.  In particular,
when $\Det{A}$ is tangible, there is a unique \nmulti\ having
highest weight.
\end{rem}

Let us review some well-known results about cycles and
multicycles.

\begin{remark}\label{rmk:graph} Given a
graph $\grph = (\mathcal V, \mathcal E)$ where $\ideg(i) \geq 1$
and $\odeg(i) \geq 1$ for each $i \in \mathcal V$, then $G$
contains a simple cycle. Indeed, otherwise $G$ must have a sink or
source, $i \in \mathcal V$, in contradiction to $\ideg(i) \geq 1$
and $\odeg(i) \geq 1$, respectively.
\end{remark}


 We also need a special case of the
celebrated theorem of Birkhoff and Von Neumann \cite{Birk}, which
states that every positive doubly stochastic $n \times n$ matrix
is a convex combination of at most $n^2$ cyclic covers; more
precisely, we quote the graph-theoretic version of {\it Hall's
marriage theorem}. Since Hall's theorem is formulated for
bipartite graphs, we note the following correspondence between
digraphs having $n$ vertices and undirected bipartite graphs
having $2n$ vertices.

\begin{rem}\label{bipart} Any digraph $\grph = ( \mathcal V, \mathcal E)$ gives rise to a bipartite graph $\widetilde \grph =  ( \widetilde{\mathcal V},\widetilde {\mathcal E})$
 whose vertex set is $\widetilde V = \tV \cup \tV'$, where $\tV'$ is
a disjoint copy of $\tV$,
 and such that any edge $(i,j) \in \mathcal E$ corresponds to an edge in~$\widetilde{ \mathcal E}$ from $i\in \mathcal V$ to $j \in \mathcal V'$.
 (Thus, the directed edges in $\grph$ correspond to undirected edges in $\widetilde \grph$.)
 \end{rem}

\begin{theorem} [\it \textbf{Hall's marriage theorem}] \label{HMT} Suppose
$\widetilde \grph = (\widetilde  {\mathcal V}, \widetilde
{\mathcal E})$ is an (undirected) bipartite graph, and for each
$j\in \widetilde{\mathcal V}$ define $$N(j) = \{ i\in
\widetilde{\mathcal V}:\text{ there is an edge in }\widetilde
{\mathcal E}\text{ connecting }i \text{ and } j\}.$$ For $\tS
\subset \widetilde{\mathcal V},$ define $N(\tS) = \cup \{ N(s): s
\in \tS \},$ and assume that $|N(\tS)| \geq |\tS|$ for every $\tS
\subseteq \widetilde{\mathcal V}.$ (Here $|\tS|$ denotes the order
of the set $\tS$.) Then, for each $k \le n,$ $\widetilde  \grph$
contains a set of edges
$$\{(\pi(1),1), \dots, (\pi(k),k)\}$$ for some $\pi \in S_n$.
(For $k=n,$ this is called a {\it matching}).\end{theorem}

  A quick proof can be
found in \cite[Theorem~2.1.2]{Diestel9} or \cite{PlM}.
 This hypothesis  provides the next lemma, motivated by an argument founded in \cite{PlM2}:

\begin{lemma}\label{graph1} Assume that $\grph = (\mathcal V, \mathcal E)$ is a digraph, possibly with
multiple edges. Then $\grph$ contains an \nmulti, under any of the
following conditions (for any $k \ge 1$ in (i) and $k>1$ in the
other parts):
\begin{enumerate} \eroman
    \item  $\ideg(j)=\odeg(i) =k $ for all $i,j$. \pSkip

\item   $\ideg(j)= k $ for   all vertices $j$ except one (at most)
with in-degree $k+1$ and one with in-degree $k-1$, and
 $ \odeg(i) = k $ for   all vertices $i$. \pSkip

\item   $\odeg(i)= k $ for   all vertices $i$ except
one (at most) with out-degree $k+1$ and one with out-degree $k-1$,
and
 $ \ideg(j) = k $ for   all vertices $j$. \pSkip

\item    $\odeg(i)= k $ for   all vertices $i$ except
one (at most) with out-degree $k-1$, and
 $\ideg(j)= k $ for   all vertices $j$ except one (at most)
with in-degree $k-1$.
\end{enumerate}
\end{lemma}
\begin{proof} We form a matrix $B$ whose $(i,j)$-entry is the number of (directed) edges from $i$ to $j$ in $\grph$,
and a new bipartite graph $\widetilde \grph$ obtained from the
graph $\grph$ as in Remark~\ref{bipart}. Thus, any nonzero entry
$b_{i,j}\in B$ corresponds to  $b_{i,j}$ edges from $i\in \mathcal
V$ to $j \in \mathcal V'$.

Note that any matching in $\widetilde \grph$ corresponds to   an
\nmulti\ of $\grph$. Thus, we  need to verify the hypothesis of
Hall's marriage theorem on $\widetilde{ \grph}$. For any $\tS
\subseteq \widetilde{\mathcal V} =  \mathcal V \cup \mathcal V'$,
write $\tU = N(\tS).$ We need to show that $|\tU| \ge |\tS|$.
First of all, since by definition the neighbors of $\mathcal V$
are in $\mathcal V'$ and visa versa, it suffices to assume $\tS
\subseteq \mathcal V $ or $\tS \subseteq \mathcal V'$.

\begin{enumerate} \eroman
    \item

By symmetry, we assume that $\tS \subseteq \mathcal V $. Then $\tU
\subseteq \mathcal V' $ and
\begin{equation}\label{momentoftruth}
k |N(\tS)| = k|\tU| = \sum _{j \in \tU} \ideg(j) = \sum _{j \in
\tU}\sum _{i \in N(j)} b_{i,j} \ge \sum _{j \in \tU}\sum _{i \in
\tS} b_{i,j} = \sum _{i \in \tS}\odeg(i) = k| \tS|,
\end{equation}
implying $|N(\tS)| \ge | \tS|,$ as desired. \pSkip

\item We modify the argument of (i), noting that if $a$ and $b$ are
integers with $a> b-1$ then $a \ge b.$ First assume that $\mathcal
S \subseteq \mathcal V '$. For any subset $\tS$ of $\mathcal V'$,
the number~$t$ of edges (counting multiplicities) terminating in a
vertex in $\tS$ is at least $(|\tS|-1)k +1.$ But since any such
edge starts at a vertex in $N(\tS)$, we see that
  $t \le |N(\tS)|k$, so we conclude that $|N(\tS)|>
|\tS|-1,$ and thus $|N(\tS)|\ge |\tS|,$ as desired.

Now assume $\tS \subseteq \mathcal V $. For any subset $\mathcal
S$ of $\mathcal V$, the number~$t$ of edges (counting
multiplicities) starting in a vertex in $\tS$ is $|\tS|k .$ But
since any such edge starts at a vertex in $N(\tS)$, we see that
  $t \le |N(\tS)|k +1$, so  again we conclude that $|N(\tS)|>
|\tS|-1,$ and thus $|N(\tS)|\ge |\tS|,$ as desired. \pSkip

\item  As in (ii). \pSkip

\item  Again the analogous argument holds. By symmetry, we assume
that $\tS \subseteq \mathcal V $. Now Equation~
\eqref{momentoftruth} becomes
\begin{equation}\label{realmomentoftruth} k |N(\tS)| =
k|\tU| \ge \sum _{j \in \tU} \ideg(j) = \sum _{j \in \tU}\sum _{i
\in N(j)} b_{i,j} \ge \sum _{j \in \tU}\sum _{i \in \tS} b_{i,j} =
\sum _{i \in \tS}\odeg(i) = k| \tS|-1,
\end{equation}
so again $|\tU| \ge |\tS|.$

\end{enumerate}
\end{proof}

\begin{prop}\label{multicyc1} Assume that $\grph = (\mathcal V, \mathcal E)$ where each vertex $i \in \mathcal V$
has $\ideg(i) = \odeg(i) = k $. Then $\grph$ is a union of $k$
distinct \nmulti s.  \end{prop}
\begin{proof} By the lemma, we have an \nmulti\, which we may remove from  $\grph
$; we thereby obtain a graph where each vertex $i \in \mathcal V$
has $\ideg(i) = \odeg(i) = k-1$, and continue by induction on $k$.

\end{proof}

\section{Quasi-invertible matrices and the adjoint}

\begin{defn} A \textbf{\quasi-zero} matrix
$\zm_\tG$ is a matrix equal to $\rzero$ on the diagonal, and whose
off-diagonal entries are ghosts or $\rzero$. (Despite the
notation, the \quasi-zero matrix $\zm_\tG$ is not unique, since
the $\nu$-values of the ghost entries may vary.) A
\textbf{\quasi-identity} matrix $\um_\tG$ is a \regular ,
multiplicatively idempotent matrix equal to $\um + \zm_\tG$, where
$\zm_\tG$ is a \quasi-zero matrix.


A matrix $B$ is a \textbf{\quasi-inverse} for $A$ if $AB$ and $BA$
are \quasi-identities. The matrix $A$ is
\textbf{\quasi-invertible} when $A$ has a \quasi-inverse.
\end{defn}

Thus, for any matrix $A$ and any \quasi-identity, $I_\tG,$ we have
$AI_\tG = A + A_\tG,$ where $A_\tG \in M_n(\tGz)$. Also, $\Det{
I_\tG } = \rone$ by the \regular ity of $I_\tG$. Note that the
identity matrix $I$ is itself a \quasi-identity, and also is a
\quasi-inverse for any \quasi-identity.

\begin{remark} $ $
\begin{enumerate} \eroman
    \item
 By definition, each \quasi-identity ${{I_\tG}}$
is also \quasi-invertible, since ${{I_\tG}}$ is a \quasi-inverse
of itself. Recall from semigroup theory that there is a one-to-one
 correspondence between (multiplicative) idempotent matrices in  $M_n(R)$ and
 maximal (multiplicative) subgroups of $M_n(R);$ the  idempotent matrix $I_\tG \in M_n(R)$  is the
 identity element of a unique maximal subgroup of $M_n(R)$, namely the group of units of
 $I_\tG M_n(R)I_\tG$;
 cf.~\cite{Lallement}.  Note that $M_n(R)$ has many other
idempotents, \regular \ and singular. \pSkip

\pSkip \item  Any \quasi-identity matrix $I_\tG = (a_{i,j})$ must
satisfy $a_{i,j} a_{j,i} < \rone^\nu$ for $i\ne j$ and
$a_{i,j}^\nu a_{j,k}^\nu \le a_{i,k}^\nu$ for $i\ne k,$ because
$I_\tG $ is multiplicatively idempotent.

\pSkip \item   A slightly weaker notion, called
\textbf{pseudo-identity}, is given in
\cite{IzhakianRowen2009TropicalRank}. Note that a pseudo-identity
need not be multiplicatively idempotent, as seen by considering
upper triangular $3\times 3$ matrices with ghost entries on the
upper diagonal (cf.~Example~\ref{uppertr} below); these do not
necessarily satisfy the criterion $a_{i,j}^\nu a_{j,k}^\nu \le
a_{i,k}^\nu$ of (ii).
\end{enumerate}

\end{remark}

There is another formula to help us out.
\begin{defn} The $(i,j)$-\textbf{minor} $A'_{i,j}$ of a matrix $A =(a_{i,j})$ is obtained by
deleting the $i$ row and $j$~column of~$A$. The \textbf{adjoint}
matrix $\adj{A}$ of $A$ is defined as the transpose of the matrix
$(a'_{i,j}),$ where $a'_{i,j}= \Det{A'_{i,j}}$.
\end{defn}

\begin{rem}\label{rmk:adj0} By definition, $a'_{i,j}$  can be computed
as \begin{equation}\label{adjform}\sum _{\pi \in S_n, \ \pi(i) =
j} a_{1,\pi(1)}a_{ 2,\pi(2)}\cdots
a_{{i-1},\pi({i-1})}a_{{i+1},\pi({i+1})}\cdots a_{{n},\pi({n})}.
\end{equation}\end{rem}
\newpage

\begin{rem}\label{rmk:adj}  $ $
\begin{enumerate} \eroman
    \item

Suppose $A = (a_{i,j})$. An easy calculation using Formula
\Ref{det2} yields
\begin{equation}\label{det3}
\Det{A} = \sum_{j=1}^n a_{i,j} \, a'_{i,j}, \quad \forall
i.\end{equation} Consequently, $(a_{i,j}  \, {a'}_{i,j})^\nu \le
\Det{A}^\nu$ for each $i,j$.\pSkip

\pSkip
\item  If we take $k\ne i,$ then replacing the $i$ row by the $k$
row in $A$ yields a matrix with two identical rows;  thus, its
\tdet \ is a ghost, and we thereby obtain
\begin{equation}\label{det3.2}
 \sum_{j=1}^n a_{i,j} \, a'_{k,j} \in \tGz, \qquad \forall
k \ne i;\end{equation}
Likewise
\begin{equation*}\label{det3.21}
 \sum_{j=1}^n a_{j,i} \, a'_{j,k} \in \tGz, \qquad \forall
k \ne i.\end{equation*}\pSkip

 More generally, by Remark~\ref{detrep}, if  $b'_{i,j}\in R$ with the same
$\nu$-value as $a'_{i,j},$ then
\begin{equation*}\label{det3.20} \sum_{j=1}^n a_{i,j} \,
b'_{k,j} \in \tGz, \qquad \forall k \ne i\end{equation*} (since
this is the tropical determinant of a matrix having two rows with
the same $\nu$-values); likewise,
\begin{equation}\label{det3.23}
 \sum_{j=1}^n a_{j,i} \, b'_{j,k} \in \tGz, \qquad \forall
k \ne i.\end{equation}

This observation is significant since it is often useful to take
$b'_{i,j}\in \tT$. The same argument shows that if $b_{i,j}\in R$
with the same $\nu$-value as $a_{i,j},$ then
\begin{equation*}\label{det3.200} \sum_{j=1}^n b_{i,j} \,
a'_{k,j} \in \tGz, \qquad \forall k \ne i.\end{equation*}
\pSkip

\end{enumerate}
\end{rem}

\begin{defn}\label{quas} For $\Det{ A}$ is invertible, define $$I_A = A \frac{\adj{A}}{\Det{A}},\qquad  I'_A = \frac{\adj{A}}{\Det{A}}A.$$
\end{defn}

Putting together (i) and (ii) of Remark~\ref{rmk:adj} shows that
the matrices $ \um _ A$ and $\um '_ A$ are the identity on the
diagonal and ghost off the diagonal.

\begin{example}\label{2by2} Let us compute $\adj{AB}$, for any $2\times
2$ matrices $$A = \( \begin{matrix} a_{1,1} & a_{1,2} \\
a_{2,1} & a_{2,2}
\end{matrix}\) ,\qquad B = \( \begin{matrix} b_{1,1} & b_{1,2} \\
b_{2,1} & b_{2,2}
\end{matrix}\) ,$$ and compare it to $\adj{B}\adj{A}$.
First, $\adj{A} = \( \begin{matrix} a_{2,2} & a_{1,2} \\
a_{2,1} & a_{1,1}
\end{matrix}\) ,$ $\adj{B} = \( \begin{matrix} b_{2,2} & b_{1,2} \\
b_{2,1} & b_{1,1}
\end{matrix}\) ,$
so $$\adj{B}\adj{A} = \( \begin{matrix} b_{2,2}a_{2,2} + b_{1,2}a_{2,1} & b_{2,2}a_{1,2}+
 b_{1,2}a_{1,1} \\
b_{2,1}a_{2,2}+ b_{1,1}a_{2,1} &  b_{2,1}a_{1,2}+b_{1,1}a_{1,1}
\end{matrix}\) ,$$
which equals $\adj{AB}$

However, for larger $n$, this fails; for example, for the $3\times
3$ matrix $$A = \( \begin{matrix} \one & \one & \one  \\
 \one & \zero & \zero   \\ \one & \zero & \zero
\end{matrix}\), \quad  \text{we have} \quad A^2 = \( \begin{matrix} \one^\nu & \one & \one  \\
 \one & \one & \one   \\ \one & \one & \one
\end{matrix}\) \ \text{and} \   \adj{A^2} = (\one^\nu),$$
 whereas
$$\adj{A} = \( \begin{matrix} \zero & \zero & \zero  \\
 \zero & \one & \one   \\ \zero & \one & \one
\end{matrix}\) \quad  \text{and} \quad \adj{A}^2 = \( \begin{matrix} \zero & \zero & \zero  \\
 \zero & \one^\nu & \one^\nu   \\ \zero & \one^\nu & \one^\nu
\end{matrix}\).$$

%
\end{example}

One does have the following fact, which illustrates the subtleties
of the supertropical structure:

\begin{prop}  $\adj{AB} =\adj{B}\adj{A}+\ghost.$ \end{prop}
\begin{proof}
 Writing $AB = (c_{i,j}),$ we see that $\adj{AB} =
(c'_{j,i})$ whereas the $(i,j)$-entry of $\adj{B}\adj{A}$ is $\sum
_{k=1}^n b'_{k,i}a'_{j,k}$. Since $a'_{j,k}b'_{k,i}$ appears in
$c'_{j,i}$, we need only check that the other terms in $c'_{j,i}$
occur in matching pairs that thus provide ghosts. These are sums
of products the form $$d_{k_1,\pi(k_1)}d_{ k_2,\pi(k_2)}\cdots
d_{k_{n-1},\pi(k_{n-1})},$$ where $k_t \ne j,$ $\pi(k_{t}) \ne i$
for all $1 \le t\le n-1,$ and
$$d_{k_t,\pi(k_t)} =  a_{k_t,\ell}b_{\ell,\pi(k_t)}.$$
If the $\ell$ do not repeat, we have a term from $\adj{B}\adj{A}$.
But if some $\ell$ repeats, i.e., if we have $$d_{k_t,\pi(k_t)} =
a_{k_t,\ell}b_{\ell,\pi(k_t)}, \qquad d_{k_u,\pi(k_u)} =
a_{k_u,\ell}b_{\ell, \pi(k_u)},$$ then in computing $c'_{j,i}$ we
also have a contribution from $\sg$ where $\sg(k_t) = \pi(k_u)$
and $\sg(k_u) = \pi(k_t)$ (and otherwise $\sig = \pi$)), where we
get
$$a_{k_t,\ell}b_{\ell,\sg(k_t)}a_{k_u,\ell}b_{\ell, \sg(k_u)} =
a_{k_t,\ell}b_{\ell,\pi(k_u)}a_{k_u,\ell}b_{\ell, \pi(k_t)} =
a_{k_t,\ell}b_{\ell,\pi(k_t)}a_{k_u,\ell}b_{\ell, \pi(k_u)},$$ as
desired. \end{proof}

We  show below that the matrices $\um _ A$ and $\um' _ A$ of
Definition~\ref{quas} are quasi-identities. This requires some
preparation. Our main technique of proof is to define a
\textbf{string}  (from the matrix $A$) to be a product $\str =
a_{i_1, j_1}\cdots a_{i_k, j_k}$ of entries from $A$ and, given
such a string, to define the \textbf{digraph $G_{\str}$ of the
string} to be the graph whose edges are  $(i_1, j_1), \dots, (i_k,
j_k),$ counting multiplicities. For example, the digraph $G_\str$
of the string
$$\str = a_{1,2}a_{2,3}a_{3,1}a_{1,1}a_{2,3}a_{3,2}$$ has edge set
 $\{ (1,1),\ (1,2),\ (2,3)$ (multiplicity 2), $(3,1), \
(3,2)\}$.

\begin{thm}\label{adjeq} $ $ \eroman
\begin{enumerate}
    \item $\Det{ A \adj{A}}  = \Det{A} ^n.$ \pSkip
    \item $\Det{ \adj{A}}  = \Det{A}^{n-1}.$
\end{enumerate}
\end{thm}
\begin{proof} $ $ \pSkip (i) First we claim that $\nu(\Det{ \adj{A} }) = \nu(\Det{A}^{n-1}).$
First note that the $(i,k)$-entry of  $A \! \adj{A}$ is
$\sum_{j=1}^n a_{i,j}\, a'_{k, j}$. Hence, by definition of
tropical determinant,
\begin{equation}\label{adjdet}
\Det{ A \adj{A} }= \sum_{\pi \in S_n}  \sum_{j_1=1}^n \cdots
\sum_{j_n=1}^n   a_{1,j_1}\, a'_{\pi(1), j_1}\cdots a_{n,j_n}\,
a'_{\pi(n), j_n}.\end{equation}
 Let $\bt_1  = \Det{A}^n,$ and
  $\bt_2$ denote the right side of \eqref{adjdet}.
Clearly $\bt_2^\nu \ge  \bt_1^\nu,$ seen by taking $j_i = i$ and
$\pi = (1).$ (Noting that the diagonal entries of $  A \adj{A} $
all are $\Det{A},$ we see that $ \Det{  A \adj{A}} $ has
$\nu$-value at least that of $ \Det{ A }^n.$)

To prove the claim, it remains to show that $\bt_2^\nu \le
\bt_1^\nu.$ Viewing \eqref{adjdet} as a sum of strings of entries
of $A$, consider a string  of maximal $\nu$-value, and take its
digraph  (counting multiplicities). Any string occurs in some
\begin{equation}\label{what1} \sum_{j_1=1}^n \cdots
\sum_{j_n=1}^n   a_{1,j_1}\, a'_{\pi(1), j_1}\cdots a_{n,j_n}\,
a'_{\pi(n), j_n},\end{equation} so we can subdivide our string
into $n$ substrings, each a summand of $a_{i,j_i}\, a'_{\pi(i),
j_i}$ as $1 \le i \le n.$ In each such substring we have $n$
edges: The edge $(i,j_i)$ appears because of $a_{i,j_i},$ and
$n-1$ other edges appear in $ a'_{\pi(i), j_i}$, namely of the
form
$$a_{i'_1, j'_1}\cdots  a_{i'_{n-1}, j'_{n-1}}$$ where $\{ i'_1,
\dots, i'_{n-1}\} = \{ 1, \dots, \pi(i)\! -\! 1, \pi(i)\! +\! 1,\,
\dots, n\}$ and $\{ j'_1, \dots, j'_{n-1}\} = \{ 1, \dots, j_i\!
-\! 1, j_i\! +\! 1,\, \dots, n\}$.

In each of these $n$ substrings,  the in-degree of each vertex is
exactly one (since $j_i$ appears in $a_{i, j_i}$, and all the
other indices appear in the adjoint term $a'_{\pi(i), j_i}$); thus
the total in-degree of each vertex in any string arising from
\eqref{adjdet} is ~$n$.

The total out-degree in any substring in \eqref{what1} is:
$$\odeg(i) = \left \{
\begin{array}{ll}
  1 \text{ for each index} & \text{when } \pi(i) = i; \\
   2
\text{ for } i, \ 0 \text{ for } \pi(i), \ 1 \text{ for all } i'
\ne i,\pi(i) &  \text{when } \pi(i) \ne
i. \\
\end{array} \right.$$ Since $\pi$ is a
permutation, the total out-degree of each vertex in any string
arising from \eqref{adjdet} is $$\left( \sum _i 1\right)+ 1 -1 =
n.$$

Hence, by Proposition \ref{multicyc1}, the digraph of $A\adj{A}$
is a union of $n$ \nmulti s, each of whose weights has $\nu$-value
at most $\Det{A}$, by Remark~\ref{matchdet1}. Hence, the term
\Ref{what1} has $\nu$-value at most that of $|A|^n$, namely
$\bt_1^\nu$, as desired.

When $\Det{A}$ is tangible, there is a unique \nmulti\ $\cyc$
 of highest weight, corresponding to some permutation $\sigma \in S_n,$ and
thus the term \eqref{what1} is obtained precisely when $\cyc$ is
repeated $n$ times. This implies that   $j$ must be $\sigma(i)$ in
each leading term in \eqref{adjdet}, yielding a unique leading
term, and $\Det{ A \adj{A}} = \Det{A} ^n.$

When $\Det{A}$ is not tangible, then either our \nmulti \
 of highest weight yields a ghost term, or we have several
\nmulti s
 of highest weight, corresponding to permutations yielding
 equal contributions to $\Det{A}$; hence $\bt _1$ and $\bt _2$ are ghosts, and again we have equality.
\pSkip (ii) Recall the formula: \begin{equation}\label{adjdet2}
\Det{ \adj{A} }= \sum_{\pi \in S_n}\prod _{i=1}^n  a'_{i,
\pi(i)}.\end{equation} The  digraph for each summand has in-degree
and out-degree $(n-1)$ for each vertex (since $\pi$ is a
permutation), so we can separate it into $(n-1)$ individual
\nmulti s, each of which has weight of $\nu$-value $\le \Det{
A}^\nu,$ proving $$\nu \left(\,\Det{ \adj{A} }\,\right) \le
\nu(\,\Det{A}^{n-1}\,).$$ On the other hand, if we take a
permutation $\pi \in S_n$ attaining $ \Det{A},$ then clearly, for
each $i_0,$ $\prod _{i\ne i_0}   a_{i, \pi(i)} =
a'_{i_0,\pi(i_0)}$, implying $a_{i, \pi(i)}\,a'_{i, \pi(i)} =
\Det{A},$ and thus
$$\nu( \Det{ \adj{A} })\ge \nu\( \sum_{\pi \in S_n} \prod _{i=1}^n
a'_{i, \pi(i)}\)  = \nu\( \prod _{i=1}^n \frac {
\Det{A}}{a_{i,\pi(i)}}\)  = \nu\( \frac {\Det{A}^n}{\Det{A}}\) =
\nu\( \Det{A}^{n-1}\) .$$

 If $A$ is \regular , then $\adj{A}$ is \regular , since we have
 only one term of maximal $\nu$-value in computing $\Det{A}$ and
 thus  $\Det{\adj{A}}$, yielding
 $\Det{\adj{A}}  = \Det{ A }^{n-1}.$

  If $A$ is singular, then so is $\adj{A}$, concluding
the proof.

(Note in the important case that $R$ is a supertropical domain and
$A$ is
 \regular, the assertion of~(ii) follows at once from (i), since Theorem \ref{thm:detOfProd} implies
$$\Det{ \adj{A} } = \frac {\Det{ A \adj{A} }}{\Det{ A}} = \Det{ A }^{n-1}.)$$
\end{proof}

 In case $\Det{A}$ is invertible in $R$, we define the
\textbf{canonical \quasi-inverse} of $A$ to be
$$A^\nb  = \frac
{\rone}{\Det{A}}\adj{A}.$$
Thus $A A^\nb  =\um _ A$, and $A^\nb  A= \um '_ A$. Note that $\um
'_ A$ and $\um _ A$ may differ off the diagonal, although
$$\um _ A A = A  A^\nb  A = A \um' _ A.$$
For example, taking $A = \( \begin{matrix} 0 & 0 \\
1 & 2
\end{matrix}\) $, we have $A^\nb  = \( \begin{matrix} 0 & -2 \\
-1 & -2
\end{matrix}\) ;$ thus $A A^\nb  = \( \begin{matrix} 0 & (-2)^\nu \\
1^\nu & 0
\end{matrix}\)  $
whereas $A^\nb  A=  \( \begin{matrix} 0 & 0^\nu \\
(-1)^\nu & 0.
\end{matrix}\) .$ The following result is given in
\cite{IzhakianRowen2009TropicalRank}, with different proof.

\begin{cor}\label{thm:idm1}
When $\Det{ A}$ is invertible,
 $\Det{ \um _ A}= \rone.$\end{cor}

 Although $\um _A$ is not the
identity, we obtain other noteworthy properties from a closer
examination of the reduced digraph $\grph _A$  of $A$, and of how
it is used to compute $ A \adj
  A.$ As before, we write $A = (a_{i,j})$ and   $\adj
  A = (a'_{i,j}).$ Since the $(i,j)$ entry of $A \adj
  A $ is $  \sum
a_{i,k}\, a'_{j,k},$  we examine the terms $a_{i,k}\, a'_{j,k}$
where $i \ne j$.

The digraph $\grph_{i,j,k}$ of  $\grph_A$ corresponding to any
string appearing in $a_{i,k}\, a'_{j,k}$ has in-degree 1 at each
vertex (since $a'_{j,k}$ provides in-degree 1 at every vertex
except $k$, and $a_{i,k}$ provides in-degree 1 at the vertex
~$k$); likewise $\grph_{i,j,k}$ has out-degree 2 at $i$, 0 at $j$,
and 1 at
 each other
vertex. Let us call such a subgraph an \textbf{$n$-\pnmulti}.

Conversely, given an $n$-\pnmulti \ $\cyc$ having out-degree 2 at
$i$ and $0$ at $j$, we take $a_{i,k}$ corresponding to an edge of
$\cyc$, and note that the remaining edges correspond to some
$(n-1)$-multicycle in the graph corresponding to $a'_{j,k}$; thus
$\cyc$ provides a term of $\nu$-value at most $a_{i,k}\,
a'_{j,k}.$ (Incidentally, since the out-degree at $i$ is 2, we
have two possible choices of $k$ that provide the same
$\nu$-value, thereby giving us an alternate proof that the
off-diagonal entries of $A \adj   A $ are ghost.) Now we need
another immediate consequence of Lemma~ \ref{graph1}:

  \begin{lemma}\label{graph2} Assume that $\grph = (\mathcal V, \mathcal E)$, where each vertex $i \in \mathcal V$
has $\odeg(i) = k $, and all but two vertices have $\ideg(i) = k
$, and one vertex $i'$ has $\ideg(i') = k+1 $ and one vertex $j'$
has $\ideg(j') = k-1 $. Then $\grph$ is a union of $k-1$ \nmulti s
and an $n$-\pnmulti.
\end{lemma}
\begin{proof} By Lemma~ \ref{graph1}(iii), $\grph$ contains an
\nmulti, which we delete and then conclude by induction on~$k$.
\end{proof}

\begin{thm} \label{thm:idm0} $(A \adj{A})^2 = \Det{A}A \adj{A},$ for every matrix $A$. \end{thm}
\begin{proof} We check   that $(A \adj{A})^2 = \Det{A} A \adj{A}$
 at each entry. The $(i,j)$-entry $b_{i,j}$ of $(A
\adj{A})^2$ is
$$\sum _{k,\ell,m = 1}^n a_{i,k} a'_{\ell,k}a_{\ell,m}a'_{j,m}.$$
Taking $\ell = j$ yields $\sum _{k,m} a_{i,k}
a'_{j,k}a_{j,m}a'_{j,m}= \Det{ A} \sum _{k,m} a_{i,k} a'_{j,k}$,
proving that $b_{i,j}$ has $\nu$-value at least that of the
$(i,j)$-entry of $\Det{A}A \adj{A}$. The reverse inequality comes
from Lemma \ref{graph2}, which enables us to extract an \nmulti,
whose $\nu$-value is at most $\Det{ A}$. Clearly the off-diagonal
terms of $(A \adj{A})^2$ are ghosts; the diagonal terms are all
tangible iff $A$ is \regular , for, in that case, the tropical
determinant is tangible.
\end{proof}

\begin{thm}\label{thm:idm} When $\Det{A}$ is invertible, $A A^\nb $ and $A^\nb  A $ are
quasi-identities (not necessarily the same),  and thus
 $A^\nb $ is a
\quasi-inverse for $A$. \end{thm}
\begin{proof} This is Corollary \ref{thm:idm1} and Theorem
\ref{thm:idm0} together.
\end{proof}

\begin{rem}\label{makegen} In case $R$ is a supertropical
semifield, then $A^\nb$ has been defined whenever $|A| \in \tT.$
We can also define $A^\nb$  for $|A|\ne \rzero$ ghost by dividing
each entry of $\adj{A}$ by some tangible element whose $\nu$-value
is $|A|$. Then $A A^\nb =   \bar I_A$  where $\bar I_A$ is
$\rone^\nu$ on the diagonal and ghost off the diagonal, and
Theorem \ref{thm:idm0} now implies that $(\bar I_A)^2 = \bar I_A$
since $(\rone^\nu)^2 = \rone^\nu.$ Likewise, we can write $A^\nb A
= \bar I'_A,$ where  $\bar I'_A$ is $\rone^\nu$ on the diagonal
and ghost off the diagonal, with $(\bar I'_A)^2 = \bar I'_A$.
These observations enable us to treat singular matrices in an
analogous manner to nonsingular ones, just as long as $|A| \ne
\rzero.$
\end{rem}

One might hope that the same proof of Theorem~\ref{thm:idm0} would
yield the better result that
$$A \adj{A} A = \Det{A} A,$$ (i.e., $A A^\nb  A = A$ for $\Det{A}$
invertible), which we call the ``von Neumann regularity
condition'', cf. \cite{Lallement}. Unfortunately, this is false in
general! The difficulty is that one might not be able to extract
an $n$-multicycle from
\begin{equation}\label{vonNeum} a_{i,j}a'_{k,j}a_{k,\ell}.\end{equation} For example, when $n=3$, we have
the term
$$a_{1,1}(a_{1,3}a_{3,2})a_{2,2}= a_{1,1}a'_{2,1}a_{2,2},$$ which does not contain an
$n$-multicycle. This is displayed explicitly in the following
example (in logarithmic notation, as usual).

\begin{example}\label{uppertr}

$$ \text{Let } \ \ A = \( \begin{matrix} 10 & 0 & 10\\
 0 & 10 & 0\\  0 & 10 & 1
\end{matrix}\). \qquad \text{ \ Then  \ } \adj{A} = \( \begin{matrix} 11 & 20 & 20\\
 1 & 11 & 10^\nu\\  10^\nu & 20 & 20
\end{matrix}\) ,$$
$$A \adj{A}  = \( \begin{matrix} 21 & 30^\nu & 30^\nu\\
 11^\nu & 21 & 20^\nu\\  11^\nu & 21^\nu & 21
\end{matrix}\) ,\qquad \text{and} \ \  A \adj{A} A = \( \begin{matrix} 31 & 40^\nu & 31^\nu\\
 21^\nu & 31 & 21^\nu\\  21^\nu & 31^\nu & 22
\end{matrix}\).$$
As expected, the von Neumann regularity condition is ruined by the
(1,2) position.
\end{example}

An even easier example of the same phenomenon can be seen via
triangular matrices, again for $n\ge 3$.

\begin{example} $ $ \pSkip Take
$A = \vvMat{0}{a}{b} { }{0}{c}  { }{ }{0}$. Then $\adj{A} =
\vvMat{0}{a}{b + ac} { }{0}{c}  { }{ }{0}$, and
$$A\adj{A}A =A\adj{A}= \vvMat{0}{a^\nu}{b^\nu + (ac)^\nu} { }{0}{c^\nu}  {
}{ }{0} \neq A,$$   when $(ac)^\nu
> b^\nu$.
\end{example}

$ $ From a positive perspective, if each digraph arising from
\eqref{vonNeum} does contain an \nmulti, then the matrix $A$
satisfies the von Neumann regularity condition. In particular,
this is true when $n=2.$

Conversely to Theorem~\ref{thm:idm}, we have

\begin{prop} Each quasi-identity $I_\tG$ satisfies $\adj{I_\tG}
= I_\tG^\nb  = I_\tG,$ and thus $I_{I_\tG} = I_\tG.$\end{prop}
\begin{proof} Write $I_\tG = (a_{i,j}).$ The $(i,j)$-entry $a'_{i,j}$ of
$\adj{I_\tG}$ is the sum of those terms corresponding to a path in
$\grph _{I_\tG}$ having out-degree 0 at $i$, in-degree 0 at $j$,
and otherwise out-degree and in-degree 1 at all vertices. When $i
= j,$ then this is an $(n-1)$-multicycle, which must have weight $
\le \rone$ since $\Det{ I_\tG} = \rone,$ and we get $\rone$ from
the string $$a_{1,1}\cdots a_{i-1,i-1}a_{i+1,i+1}\cdots a_{n,n} =
\rone^{n-1} = \rone.$$

Thus it remains to check those $a'_{i,j}$ for $i\ne j.$ We need to
show that $a'_{i,j} = a_{j,i}$, which by hypothesis  is ghost.
 In computing $a'_{i,j}$, we have the term
 $$a_{j,i} \prod _{k
\ne i,j} a_{k,k} = a_{i,j} \, \rone \cdots \rone = a_{j,i} \, ,$$
implying ${a'_{i,j}}^\nu \ge {a_{j,i}}^\nu .$ But all strings in
$a'_{i,j}$ have $\nu$-value $\le {a_{j,i}},$ because they can be
decomposed as the union of cycles and a path from $j$ to $i$; the
weight of any cycle must have $\nu$-value at most $\rone^\nu$
(since $\Det{ I_\tG }= \rone$), and the weight of any path from
$j$ to $i$ has $\nu$-value most ${a_{j,i}}$ because $I_\tG$ is
idempotent. Thus, $a'_{i,j}= a_{j,i}.$
\end{proof}

We conclude that a necessary and sufficient condition for a matrix
$B$ to have the form $A A^\nb  $ is for $B$ to be a
quasi-identity. By symmetry, this is also a necessary and
sufficient condition for the matrix $B$ to have the form $ A^\nb
A$ (but possibly with different $A$).

This leads to another positive result concerning von Neumann
regularity. First we want to compare $\adj{A}$ and $\adj{A A^\nb
A}$ for $A$ nonsingular. One must be careful, since it is not
necessarily the case that
$\adj{AA^\nb A } = \adj{A};$ for example, with  $A = \( \begin{matrix} 0 & 1 \\
-\infty & 0
\end{matrix}\) $, we have $\adj{A} = A$ but $\adj{A} A = \adj{AA^\nb A} = \( \begin{matrix} 0 & 1^\nu \\
-\infty & 0
\end{matrix}\) $.
\newpage

\begin{lem}\label{adjmat} The corresponding entries of
$\adj{AA^\nb A }$ and $\adj{A}$ have the same $\nu$-values.
\end{lem}
\begin{proof}  Write $ AA^\nb
A = (b_{i,j})$ and $\adj{AA^\nb A } = (b'_{i,j})$. Since   $I_A =
I + \ghost$, clearly ${b'_{i,j}}^\nu \ge {a'_{i,j}}^\nu,$ so it
suffices to prove that ${b'_{i,j}}^\nu \le {a'_{i,j}}^\nu.$ But
$b_{i,j}$ is a product of terms $\frac \rone {\Det{A}}
a_{i,k}a'_{\ell,k}a_{\ell,j}$. For any string appearing in such a
product, $i$ has out-degree 2 and all other indices have
out-degree 1; likewise, $j$ has in-degree 2 and all other indices
have in-degree 1. Thus, in computing any string for $b'_{i,j},$
which we recall is a product
$$b_{i'_1, j'_1}\cdots b_{i'_{n-1}, j'_{n-1}}$$ where $\{ i'_1,
\dots, i'_{n-1}\} = \{ 1, \dots, i-1, i+1, \dots, n\}$ and $\{
j'_1, \dots, j'_{n-1}\} = \{ 1, \dots, j-1, j+1, \dots, n\}$, we
see that the out-degree is $n-1$ for $i$, and $n$ for all other
vertices; likewise, the in-degree is $n-1$ for $j$, and $n$ for
all other vertices. Hence, by Lemma~\ref{graph1}(iv) we can
extract $n-1$ \nmulti s, each having value $\le |A|,$ and are left
with a graph of out-degree 0 for $i$ and out-degree 1 for each
other vertex, and in-degree~0 for $j$ and in-degree 1 for each
other vertex; the product of the corresponding entries of $A$ is a
summand of $a'_{i,j}$. In other words, ${b'_{i,j}}^\nu$ is a sum
of terms, each of which is   $\rone^\nu$ times ${a'_{i,j}}^\nu,$
as desired.
\end{proof}

\begin{lem}\label{adjmat2} $\Det{ AA^\nb A }^\nu =
\Det{A}^\nu,$ for any matrix $A$ over a supertropical semifield.
\end{lem}
\begin{proof} Applying Theorem \ref{adjeq} to Lemma~\ref{adjmat},
 $$\(\Det{ AA^\nb A }^{n-1}\)^\nu =  \Det{\adj{AA^\nb A }}^\nu =
 \Det{\adj{A}}^\nu = \(\Det{ A}^{n-1}\)^\nu, $$ implying $\Det{ AA^\nb A }^\nu =
\Det{A}^\nu,$ since $\tG$ is an ordered group.
\end{proof}

\begin{prop} $A A^\nb  A $ satisfies the von Neumann regularity
property, for any nonsingular matrix~$A$ over a supertropical
semifield.
\end{prop}
\begin{proof} First we claim that $I_{AA^\nb  A } = I_{A}.$
Indeed, since $I_{AA^\nb  A }$ and $ I_{A}$ are both
quasi-identities, it suffices to show that their respective
off-diagonal entries have the same $\nu$-values (since they are
ghost, by definition). But
$$I_{AA^\nb  A } = {AA^\nb  A }{(AA^\nb A )}^\nb = \frac 1{\Det{A}} I_A A \adj{AA^\nb
A }$$ whereas $$I_{A } = I_A^2 = \frac 1{\Det{A}} I_A A \adj{A}.$$
The claim follows when we observe that the corresponding entries
of $\adj{AA^\nb A }$ and $\adj{A}$ have the same $\nu$-values, in
view of Lemma~\ref{adjmat2}.

But now, using the fact that $I_A$ is multiplicatively idempotent,
we have
$$ (AA^\nb  A ) (AA^\nb  A )^\nb  (AA^\nb  A ) = I_{AA^\nb A
}AA^\nb  A  = I_A AA^\nb  A = I_A^2 A =  I_A A = AA^\nb A.$$
\end{proof}

Here is another application of the adjoint matrix, to be
elaborated in a follow-up paper.

\begin{rem} Suppose $\Det{ A}$ is invertible, and $v \in R^{(n)}.$
Then the equation $Aw = v + \text{ghost}$ has the solution $w =
A^\nb v.$ Indeed, writing $I_A = I + Z_\tG$ for some quasi-zero
matrix $Z_\tG,$ we have $$Aw = AA^\nb v = I_A v = (I + Z_\tG)v =
v+ \text{ghost}.$$\end{rem}

\section{The Hamilton-Cayley theorem}

\begin{defn}\label{esschar}
Define the \textbf{characteristic polynomial} $f_A$ of the matrix
$A$ to be
$$ f_A = \Det{\la I+ A},  $$
the \textbf{essential characteristic polynomial} to be the
essential part ${f_A}^{\essn}$ of the characteristic polynomial~$
f_A$, cf.~\cite[Definition 4.9]{IzhakianRowen2007SuperTropical},
and the \textbf{tangible characteristic polynomial} to be a
tangible polynomial $\hat f_A = \la ^n + \sum _{i=1}^{n} \hat \a
_i \la ^{n-i} $,  where $\hat \al_i \in \tTz$ and $\hat \al_i^\nu
= \al_i^\nu$, such that $f_A = \la ^n + \sum _{i=1}^{n} \a _i \la
^{n-i}$.
\end{defn}

 Under this notation, we see that  $\a _k \in R$ is the highest weight of the
$k$-multicycles in the reduced digraph $\grph _A$ of $A$.

Recall that the \textbf{roots} of a polynomial $f \in R[\la]$ are
those elements $a \in R$  for which $f(a)\in \tGz$.
 Thus, we say that a matrix $A$
\textbf{satisfies} a polynomial $f\in R[\la]$ if $f(A) \in
M_n(\tGz).$

\begin{thm}\label{hamilton-Cayley}(\textbf{Supertropical Hamilton-Cayley}) Any matrix $A$ satisfies both its characteristic
polynomial~ $f_A$ and its tangible characteristic polynomial $\hat
f_A$.\end{thm}
\begin{proof} Let $B = \hat f_A(A)=A ^n + \sum \hat
\a_i A^{n-i}$. It suffices to prove that $B \in M_n(\tGz),$ i.e.,
that each entry $b_{u,v}$ is ghost. But $b_{u,v}$ is obtained as
the maximum from the various contributions $\hat \a_i A^{n-i},$
each of which is the product of weights of disjoint simple cycles
$\cyc_1, \dots, \cyc_{t(u,v)}$ in the reduced diagraph $\grph_A$
with each $\cyc_j$ of length $n_j,$ where $\sum_{j=1}^{t(u,v)} n_j
= i,$ multiplied by the weight of a path $\pth$ of $\grph_A$ of
length $n-i$. If this last path $\pth$ intersects one of the
cycles, say $\cyc_1,$ then we also have a path of length $n-i+n_1$
obtained by combining $\pth$ with $\cyc_1$, in which case
$b_{u,v}$ is matched by a term from $\a_{i-n_1}A^{n-i+n_1},$ and
thus is ghost. Thus, we may assume $\pth$ is disjoint from all the
cycles.  But this implies that the path $p$ traverses only $n-i$
vertices, which is the length of $p$, and thus $p$ must contain a
cycle $\cyc$ of some length $m \le n-i$ (by the pigeonhole
principle). But then $b_{u,v}$ is matched with a term from $\a
_{i-m} A^{n-i-m},$ and thus is ghost. (When $m=n-i,$ we have
$u=v,$ and $\pth$ itself is a cycle $\cyc_{t(u,u)+1}$, so we match
$b_{u,u}$ with a term from~$\Det{A}.$)

When all the $\hat \al_i$ contributing to $b_{u,v}$, and thus to
$B$, are $\rzero$, it means that the cycle of length $n$ is the
unique cycle of minimal length. In this case, we have $\hat f_A(A)
= A^n + \Det{ A} I$ is ghost.
\end{proof}

We digress for a moment to improve Theorem \ref{hamilton-Cayley}
slightly, by looking closely at its proof. Given a polynomial $f =
\al_n \lm^n + \cdots + \al_1 \lm + \al_0$, we define the
polynomial $\tilde f$ to be
$$ \tilde f = \hat \al_n \lm^{n-1} + \cdots + \hat\al_2 \lm +
\hat\al_1,$$ where $\hat \al_i \in \tTz$ and $\hat \al_i^\nu =
\al_i^\nu$.
\begin{thm} $\tilde  f_A(A)= \adj{A}+\ghost$, for any matrix $A$.
\end{thm}

\begin{proof} We first show that many entries of $$B = \tilde f_A(A) + \adj{A} = A ^{n-1} + \sum  \hat\a_i A^{n-i-1}
+ \adj{A}$$  are ghosts. The $(u,v)$-entry $b_{u,v}$ is obtained
as having the largest $\nu$-value from the various  $\a_i
A^{n-i-1},$ which is the product of weights of disjoint simple
cycles $\cyc_1, \dots, \cyc_{t(u,v)}$, with each $\cyc_j$ of
length $n_j,$ where $\sum_{j=1}^{t(u,v)} n_j = i,$ together with
the weight of a path $\pth$ of length $n-i-1$. If this last path
$\pth$ intersects one of the cycles, say $\cyc_1,$ then we also
have a path of length $n-i+n_1-1$ obtained by combining $\pth$
with $\cyc_1$, so we match $b_{u,v}$ with a term from $\hat
\a_{i-n_1}A^{n-i+n_1-1}.$ Thus, we have a ghost term unless $\pth$
is disjoint from all the cycles.  But this implies that $p$
traverses only $n-i-1$ vertices, which is its length. If $p$
contains a cycle $\cyc$ of some length $m \le n-i-1$,  then
$b_{u,v}$ is matched by a term from $\a _{i-m} A^{n-i-1-m},$ and
thus is ghost.

Thus, the only unmatched terms arise precisely when $p$ does not
contain any cycle. In this case, $p$ must have the form
$$a_{k_1,\pi(k_1)}a_{ k_2,\pi(k_2)}\cdots
a_{k_{m},\pi(k_{m})},$$ where $k_t \ne u$ and $\pi(k_{t}) \ne v$
for all $1 \le t\le m,$ and $\pi(k_t) = k_{t+1}$ for all $t<m$.
But combining this with the  cycles $\cyc_1, \dots, \cyc_{t(u,v)}$
give us one of the summands in Equation \eqref{adjform} of
Remark~\ref{rmk:adj0}, and conversely any such summand can be
matched with a disjoint union of simple cycles and some path of
this form. Thus, we have decomposed $\tilde  f_A(A)$ as $\adj{A}$
plus ghost terms.\end{proof}

\begin{note}\label{note:essential} Let us compare these two notions of characteristic polynomial.
The tangible characteristic polynomial shows us that the powers of
$A$ are tropically dependent (as defined in
Definition~\ref{tropdep1} below). But, as we shall see, the
characteristic polynomial is more appropriate when we work with
eigenvalues, and its essential monomials play a special role.

Note, however, that a monomial which is inessential with respect
to substitutions in~$R$, is not necessarily inessential with
respect to matrix substitutions  in $M_n(R)$. For example,
consider the polynomial $f= \lm^2 + \lm +2$;  the term $\lm$ is
inessential for substitutions in $R$ but essential for matrix
substitutions, seen by taking the matrix $A =
\vMat{-\infty}{1}{1}{0}$ in logarithmic notation. In this case,
$A^2 = \vMat{2}{1}{1}{2}$, so $f(A) = \vMat{2^\nu}{1^\nu }{1^\nu
}{2^\nu}$ is ghost, whereas ${f}^{\essn}(A) =
\vMat{2^\nu}{1}{1}{2^\nu}$ is not ghost. The theory runs more
smoothly  when the characteristic polynomial is essential.
\end{note}

\begin{note} We   conclude from Theorem \ref{hamilton-Cayley} that  any $2 \times 2$ matrix $A$
satisfies
$$ A^2 + \trace{ A}A + \Det{ A} I \ \in \ M_2(\tGz) \ . $$
\end{note}

Here is an easy but important special case of Theorem
\ref{hamilton-Cayley}.

\begin{defn}\label{lower} A matrix $A = (a_{i,j})$ is in \textbf{lower ghost-triangular form} if $a_{i,j} \in \tGz$ for each
$i>j.$ \end{defn}

Note that if $A$ is \regular \ and is in lower ghost-triangular
form, then its diagonal terms must all be tangible.

\begin{example} Any matrix $A = (a_{i,j})$ in lower ghost-triangular form satisfies
the polynomial $$f = \prod_{i=1}^n(\la + a_{i,i}).$$
  One way of seeing this is to replace  the $a_{i,j}$ by $\rzero$ for all $i>j,$
and apply  Theorem \ref{hamilton-Cayley}. Here is a direct
verification. $f(A) = (A+ a_{1,1}I)\cdots (A+ a_{n,n}I).$ In order
to get a non-ghost entry in $f(A)$, we need to multiply together
$n$ terms from the diagonal or above. However, the (1,1) position
in the first multiplicand starts with $a_{1,1}^\nu e_{1,1},$
(where $e_{i,j}$ denote the standard matrix units), so the first
factor must be $a_{i_1, j_1}e_{i_1 ,j_1}$ for $j_1 \ge 2.$ But the
(2,2) position in the second multiplicand starts with $a_{2,2}^\nu
e_{2,2},$ implying the second factor must be $a_{i_2, j_2}e_{i_2,
j_2}$ for $j_2 \ge 3.$ Continuing in this way, we see that the
$(n-1)$-factor must be $a_{i_{n-1}, j_{n-1}}e_{i_{n-1}, j_{n-1}}$
for $j_{n-1} \ge n,$ in which case the last factor must be a
ghost.
\end{example}

\section{Applications to supertropical linear algebra}

In this section, we see how \permanent s apply to vectors over  a
supertropical domain $R$. Our main objective is to characterize
singularity of a matrix $A$ in terms of tropical dependence of its
rows.

First we start with a special case, where $A$ is \ssingular, i.e.,
$\Det{ A } = \rzero.$ In view of Remark~\ref{rmk:zeroDet}, the
answer is a consequence of results in classical matrix theory, but
anyway the statement and proof in this case are rather
straightforward, so we present it here in full.

\begin{defn}\label{rdef} We say that a set $v_1, \dots, v_k$ of vectors has
\textbf{rank defect} $\ell$ if there are $\ell$ columns, which we
denote as $j_1, \dots, j_\ell$, such that $v_{i,j_u} = \rzero$ for
all $1 \le i \le k$ and $1\le u \le \ell$.
\end{defn}

For example, the vectors $(2,\rzero,2,\rzero),
(\rzero,\rzero,\rzero,2), (1,\rzero,\rzero,\rzero)$ have rank
defect 1, since they are all $\rzero$ in the second column.

\begin{prop}\label{ssing} An $n\times n$ matrix $A$ has tropical determinant $\rzero$,
iff, for some $1 \le k\le n$, $A$ has $k$ rows having rank defect
$n+1-k.$\end{prop}
\begin{proof} $(\Leftarrow)$ If $k=n$ then this is obvious, since
some column is entirely $\rzero$. If $n>k$, we take one of the
columns $j$ other than $j_1, \dots, j_k$ of Definition~\ref{rdef}.
Then for each $i$, the $(i,j)$-minor $A_{i,j}$ has
  at least $k-1$ rows with rank defect
$(n-1)+1-k$, so has tropical determinant $\rzero$ by induction;
hence $\Det{ A } = \rzero,$ by Formula~\eqref{det3}.

$(\Rightarrow)$ We are done if all entries of $A$ are $\rzero,$ so
assume for convenience that $a_{n,n}\ne \rzero$. Then the minor
$A_{n,n}$ has tropical determinant $\rzero$, so, by induction,
$A_{n,n}$ has $k\ge 1$ rows of rank defect $$(n-1)+1-k = n-k.$$

For notational convenience, we assume that $a_{i,j} = \rzero$ for
$1 \le i \le k$ and $1 \le j \le n-k.$ Thus, we can partition $A$
as the matrix $$ A = \( \begin{matrix} \zero & B' \\
B'' & C
\end{matrix}\) ,$$
where $\zero$ denotes the $k \times n\! -\! k$ zero matrix, $B'$
is a $k\times k$ matrix, $B''$ is an $n\!-\! k \times n\! -\! k$
matrix, and $C$ is an $n\!-\! k \times k$ matrix.

By inspection, $\Det{B' }\Det{B'' } = \Det{A } =\rzero$; hence
$\Det{B' }= \rzero$ or $\Det{B'' } = \rzero$. If $\Det{B' } =
\rzero$,
 then, by induction, $B'$ has $k'$ rows of rank defect
$k+1-k',$ so altogether, the same $k'$ rows in $A$ have rank
defect $(n-k) + k+1-k' = n+1-k',$ and we are done taking $k'$
instead of $k$.

If $\Det{B'' } = \rzero$,
 then, by induction, $B''$ has $k''$ rows of rank defect
$(n-k)+1 -k'',$ so altogether, these $k+k''$ rows in $A$ have rank
defect $n+1-(k+k''),$ and we are done, taking $k+k''$ instead of
$k$.
 \end{proof}

 Now we turn to the supertropical version, whose statement has quite a different
 flavor of linear dependence.

\begin{defn}\label{tropdep1}
Suppose $V = (R^{(n)},\tHinf, \mu)$ is a module  over a
supertropical semiring~$R$. A subset $W \subset V$ is
\textbf{tropically dependent} if there is a finite sum $\sum \a_i
w_i \in \tHinf$, with each $\a_i \in \tTz$, but not all of them
$\rzero$; otherwise $W\subset V$ is called \textbf{tropically
independent}.
\end{defn}

\begin{thm}\label{dep2}(See \cite[Corollary 3.3]{zur05TropicalRank} and \cite[Theorem 2.6]{IzhakianRowen2009TropicalRank}) If vectors
$v_1, \dots, v_n \in R^{(n)}$ are tropically dependent, for $R$ a
supertropical domain, then $\Det{v_1, \dots, v_n} \in \tGz$.
\end{thm}

\begin{proof}
Our proof follows the lines of \cite[Theorem
 2.6]{IzhakianRowen2009TropicalRank}. Let $A$ be the matrix whose $i$-th row is $v_i$.
Thus, writing $v_i = (a_{i,1}, \dots, a_{i,n}),$ we have $A =
(a_{i,j})$.
 We need to prove that $\Det{A}$ is ghost, so for the
remainder of the proof, we assume on the contrary that $\Det{A}$
is tangible, and aim for a contradiction.

 Rearranging the rows and
columns does not affect linear dependence of the rows, so we may
assume that $\Det{A}$ is attained  by the identity permutation,
i.e., $\Det{A} = a_{1,1}\cdots a_{n,n},$ and  is not attained by
any other permutation.

 We are given some dependence $\sum \a _i v_i \in
\tHinf$. First assume that $\a_n = \rzero;$ i.e.,
$\sum_{i=1}^{n-1} \a _i v_i \in \tHinf$. If we erase say the $j$
column of the $v_i$'s, we are left with the minor $A'_{n-1,j}$
whose rows clearly satisfy the same dependence then  by induction,
its tropical determinant $a'_{n-1,j} \in \tGz ,$ so $$|A| = \sum
_{j=1}^n a_{n-1,j} a'_{n-1,j}\in \tGz ,$$ and we are done. Thus,
we may assume that every $\a_n \ne \rzero.$

 Replacing $v_i$
by $\a _i v_i $ for $1\le i \le n,$ with
 $\al_i$ tangible, we may assume that \begin{equation}\label{sumg} \sum  v_i \in \tHinf \ .\end{equation}

We say $a_{i,j}\in A$ is (column) \textbf{critical} if
$a_{i,j}^\nu \ge a_{i',j}^\nu$ for each $1 \le i' \le n;$ in other
words, if $a_{i,j}$ dominates all entries in the $j$ column of
$A$.  Note that for this particular matrix $A$, any critical entry
is either
 ghost, or is matched by another critical entry in the same
 column.

Let $\grph_A$ denote reduced digraph of $A$,   let $\grph'$ denote
the sub-digraph of edges corresponding to critical entries, and
let $\grph''$ denote the sub-digraph of $\grph'$ after we erase
all the loops of $\grph'$. (The loops correspond to critical
diagonal elements $a_{i,i}$.)

 Note that if some $a_{i,i}\in \tGz$ then
$\Det{A} \in \tGz$, and we are done. Thus, any critical diagonal
entry must be tangible, and thus must be matched by another
critical entry in the same column. It follows that  $\grph''$ has
in-degree $\ge 1$ in each vertex, so Remark \ref{rmk:graph}
implies that
 $\grph''$ contains a cycle (which by definition of $\grph''$ is not a loop);
 this corresponds to $$a_{i_1, i_2}\cdots a_{i_{k-1},i_k} a_{i_k,i_1}$$
where each entry  is critical. Defining the permutation $\pi$ by
$\pi (i_1) = i_2, \dots, \pi(i_k) = i_1$ and the identity
elsewhere, it is clear that $a_{i_1, i_1}\cdots
a_{i_{k-1},i_{k-1}} a_{i_k,i_k}$ is dominated by $a_{i_1,
i_2}\cdots a_{i_{k-1},i_k} a_{i_k,i_1}$, and thus $\Det{A}^\nu$ is
also attained by $\pi,$ contrary to $\Det{A} \in \tT.$
\end{proof}

We look for the converse of Theorem \ref{dep2}.

\begin{thm}\label{thm:base}(See \cite[Corollary 3.3]{zur05TropicalRank}
and \cite[Theorem 2.10]{IzhakianRowen2009TropicalRank})  Suppose
$R$ is a supertropical domain. Vectors $v_1, \dots, v_n \in
R^{(n)}$ are tropically dependent, iff $\Det{ A } \in \tGz$, where
$A$ is the matrix whose rows are $v_1, \dots, v_n$. Furthermore,
we explicitly display the tropical dependence in the proof.
\end{thm}

\begin{proof}
\noindent $(\Rightarrow)$ By Theorem \ref{dep2}.

${ \bf (\Leftarrow)}$ Assuming that $A$ is singular, we need to
prove that the rows of $A$ are tropically dependent. Arguing  by
induction $n$,  we assume that the theorem is true for $(n-1)$,
the case for $n =1$ being obvious.

Rearranging the rows and   columns of $A$, we
 assume henceforth that the identity permutation $\pi = (1)$
 attains $\Det{A}.$  Note that this hypothesis is
not affected by multiplying through any row by a given tangible
element, which we do repeatedly throughout the proof.

 Let
$$\gm _\pi = v_{ \pi(1),1 } \cdots v_{ \pi(n),n }$$
for each permutation $\pi$ of $\{1, \dots, n \}$, and let
$$\gm  = \gm_{(1)} = v_{1,1} \cdots v_{n,n}.$$
 Thus $\gm^\nu = |A|^\nu = |A|.$

\textbf{Case I:} $\gm ^\nu = \gm _\pi ^\nu$ for some permutation
$\pi \ne (1).$ Thus, $\pi (i_0) \ne i_0$ for some $i_0$; for
notational convenience, we assume that $\pi(1) \ne 1.$ Take
$\beta_i\in \tTz$ of the same $\nu$-value as the tropical
determinant $|A_{i,1}|$ of the minor $A_{i,1}$. Then $\sum_{i=1}^n
\beta _i a_{i,1}$ has the same
 $\nu$-value as $\sum |A_{i,1}|a_{i,1} = |A|,$ but is ghost
 since, by hypothesis, there are two leading summands in the determinant formula that
 match. Hence,  $\sum_{i=1}^n \beta _i a_{i,1}   \in \tGz.$
 On the other hand, for every $j \ne 1,$ $\sum_{i=1}^n \beta _i
 a_{i,j}$ is the tropical determinant of a matrix having two columns with
 the same $\nu$-values, so is   in $\tGz$ by Equation~\eqref{det3.23}.
Thus, we are done unless all $\beta_i = \rzero.$ In this case
$\gamma
 =\rzero,$ so in view of Proposition~\ref{ssing}, there is $k$ for which
 $A$ has  $k$ rows with rank defect $n+1-k$. We need to conclude that
 these $k$ rows are tropically
 dependent. By induction on $n$, we may assume that $n=k+1$, and
 that the first entry of each row is $\rzero.$ If  $|A_{1,1}| \ne \rzero,$ we
 are done by the above argument.  If $|A_{1,1}| = \zero,$ we see by induction that $v_2,
 \dots, v_n$ are tropically dependent.

\textbf{Case II:} $\gm ^\nu > \gm _\pi ^\nu$ for each permutation
$\pi \ne (1).$ Thus $\gamma = |A|\in \tGz$, so some $a_{i,i}\in
\tGz$; renumbering the indices, we may assume that $a_{1,1}\in
\tGz.$ As in Case I, take $\beta_i\in \tTz$ of the same
$\nu$-value as $|A_{i,1}|$.  Then $\sum_{i=1}^n \beta _i a_{i,1}$
has the same
 $\nu$-value as $\sum |A_{i,1}|a_{i,1} = |A|,$ but is ghost
 since by hypothesis $a_{1,1}\in
\tGz.$ Again, by Equation~\eqref{det3.23},
 $\sum_{i=1}^n \beta _i a_{i,j}   \in \tGz,$ for all
$j \ne 1.$ Thus, $\sum_{i=1}^n \beta _i v_i   \in \tHinf,$ as
desired.
\end{proof}

\begin{cor}\label{coldep} (See \cite[Corollary 3.3]{zur05TropicalRank} and \cite[Theorem 3.4]{IzhakianRowen2009TropicalRank}) The matrix $A\in M_n(R)$
over a supertropical domain $R$  is \regular \ iff the rows of $A$
are tropically independent, iff the columns of $A$ are tropically
independent.
\end{cor}
\begin{proof} Apply the theorem to $\Det{A}$ and  $\Det{A^t}$,
which are the same.
\end{proof}

\begin{cor} Any $n+1$ vectors in $R^{(n)}$ are tropically
dependent. \end{cor}
\begin{proof} Expand their matrix to an $(n+1) \, \times \,(n+1)$
matrix $A$ by adding a column of zeroes at the beginning;
obviously $A$ is \ssingular, so its rows are tropically dependent.
\end{proof}

As pointed out in \cite[Observation 2.6]{zur05TropicalAlgebra},
and  as we have seen in Example \ref{2by2} above, the square of a
\regular \ matrix $A$ need not be \regular.

\subsection{The Vandermonde matrix}

One way of applying this method is by means of a version of the
celebrated Vandermonde argument. Given $a_1, \dots, a_n$ in $R$,
define the \textbf{Vandermonde matrix} $A$ to be the $n \times n$
matrix $(a_{i,j})$, where $a_{i,j} = a_i^{j-1}$ and $a_i^{0} =
\rone$. Recall from \cite[Lemma
7.58]{IzhakianRowen2007SuperTropical} that its tropical \tdet \ is
\begin{equation}\label{eq:Van} \Det{ A} =
\prod_{i\neq j } (a_i +a_j).
\end{equation}
\begin{remark}
Assume that $A $  is a Vandermonde matrix $(a_i^{j-1})$ with
respect to distinct $a_1, \dots, a_n$. By Formula~\Ref{eq:Van}, we
see that if all the $a_i$ are tangible, or if the only $a_i$ which
is ghost
 is the $a_i$
 of smallest $\nu$-value, then $A$~is \regular ; otherwise
$A$ is singular.
\end{remark}

\begin{lem}\label{iso2} If $A \in M_n(R)$  and  $v$ is
 a tangible vector such that $Av$ is a ghost vector,
then the matrix $A$ is singular.
\end{lem}
\begin{proof} The columns of $A$ are tropically dependent, so $A$
is singular by Corollary~\ref{coldep}.
\end{proof}
\begin{thm}\label{Van1} Suppose $v = (\gamma _1, \dots, \gamma _n) \in R^{(n)}$ for
$R = \RGnu$ a supertropical domain,  and suppose $\sum_{j=1}^n
a_i^j \gamma _j \in \tGz ^{(n)}$ for each $i = 1, \dots, n,$ where
$a_1, \dots, a_n$ are tangible.  Then some $\gamma _j$ is ghost.
\end{thm}
\begin{proof} Let $A$ be the Vandermonde matrix $(a_i^{j-1}).$
Then $Av$ is ghost, so we are done by the lemma.
\end{proof}

\begin{example}\label{Van3} Despite these nice applications of the Vandermonde
matrix,
 the Vandermonde matrix $A = \( \begin{matrix} 0 & 0 \\
1 & 2
\end{matrix}\) $ (over~$D(\Real)$)
has the poor behavior that $A^2 = \( \begin{matrix} 1 & 2 \\
3 & 4\end{matrix}\) ,$ which is singular with tropical determinant
$5^\nu$ whereas $\Det{ A } = 2; $ cf.~Example \ref{2by2}.
\end{example}

\begin{definition}\label{def:conjugate}
 A matrix $B_1$ is  (classically) \textbf{conjugate to} $B$ if $B_1
=A^\nb  B A$ for some    matrix $A$ with $\Det{A}$ invertible in
$R$. More generally, a matrix $B_1$ is
  \textbf{tropically conjugate to} $B$ if $B_1
=A^\nb  B A +\text{ghost}$ for some  \ matrix~$A$  with $\Det{A}$
invertible.
\end{definition}

\begin{lem}
If $f\in R[\la]$ is a polynomial with constant term $\fzero$. Then
for any \regular \ matrix $A$, $$f( A^\nb  B A) =  A^\nb  f(B) A +
\text{ghost} \ . $$
\end{lem}
\begin{proof} It is enough to check the case that $f = \la^i$ for $i\ge 1.$
Assume $B_1 = A^\nb  B A $. Let $\um _ A = A A^\nb = (I +
 Z_\tG)$, where $Z_\tG$ is a quasi-zero matrix. For any $i > 0$,
 $$ (A^\nb  B A )^i =  A^\nb  B (I +
 Z_\tG) B  \cdots B(I +
 Z_\tG) B A  =  A^\nb  B^i A + \text{ghost} \ . $$
\end{proof}

\begin{proposition}

If $B$ satisfies a polynomial $f \in R[\lm]$, $R$ is a
supertropical domain, then every tropical conjugate of $B$
satisfies $f$.
\end{proposition}

\begin{proof} It is enough to show that every conjugate
of $B$ satisfies $f$, since the added ghost only yields extra
ghost terms. Writing $f = g + \al_0$, where $g$ has constant term
$\fzero$, we have
$$ f(A^\nb  B A ) = A^\nb  g(B) A + \text{ ghost}+ \al_0 I ,$$
whereas $A^\nb  g(B) A + \al_0 A^\nb  A= A^\nb  f(B) A$ is ghost.
Write $g(B) = (b_{i,j}).$ The diagonal terms of~$ f(A^\nb  B A )$
are ghost, since they are ghosts plus the diagonal terms of
$f(B)$, which by hypothesis is ghost. Thus, we need only check the
off-diagonal terms of $A^\nb  g(B) A$, which when multiplied by
~$\Det{ A }$ have the form $\sum _{j,k} a'_{j,i}\, b_{j,k}\,
a_{k,\ell},$ for $i \ne \ell;$ we need to show that these are
ghosts.

On the other hand, $f(B) = g(B) + \a_0 I,$ so $f(B)$ and $f(A^\nb
B A )$ agree off the diagonal. When $j\ne k$, $b_{j,k}$ is either
ghost or is the same as the $(j,k)$-entry of $f(B)$, which is
ghost by hypothesis, so we may assume that $j=k$. Now, when
tangible,
$$\sum _j a'_{j,i} b_{j,j} a_{j,\ell} = \sum _j a'_{j,i} \a_0
a_{j,\ell} = \a_0 \sum _j a'_{j,i} a_{j,\ell}, $$ which is ghost
by
 \Ref{det3.2}.
\end{proof}

\section {Supertropical  eigenvectors}

We work throughout with matrices over a supertropical semifield
$F$.
\begin{defn}\label{eigen0} A vector $v$ is an \textbf{eigenvector} of $A$,
with \textbf{eigenvalue} $\bt$, if $Av = \bt
v$. The eigenvalue ~$\bt$ with $\bt^\nu$ ~maximal is said to be of
\textbf{highest weight}. \end{defn}

Definition \ref{eigen0} is standard (not requiring the language of
ghosts), and indeed it is known \cite{brualdi} that any (tangible)
matrix has an eigenvalue of highest weight. However, even counting
multiplicities, the number of eigenvalues often is less than the
size of the matrix, since certain roots of the characteristic
polynomial may be ``lost" as eigenvalues.

\begin{example}\label{eigen}  The characteristic polynomial $f_A$
of $$A =\vMat{4}{0}{0}{1}$$ over $F = D(\Real)$,
 is $(\la+4)(\la+1) +
0 = (\la+4)(\la+1),$ and indeed the vector $(4,0)$ is a
 eigenvector of $A$, with eigenvalue 4. However, there is no
eigenvector having eigenvalue 1.
\end{example}

We rectify this deficiency by weakening Definition \ref{eigen0}.
Actually, there are several possible definitions of supertropical
eigenvalue. We present two of them; the second is stronger but
suffices for our theory, so we call the first one ``weak.''

\begin{defn}\label{eigen3} A  vector $v\ne (\zero)$ is a \textbf{weak generalized supertropical  eigenvector} of $A$,
with (tangible) \textbf{weak generalized supertropical eigenvalue}
$\bt \in \tTz$,
 if $A^m v + \bt^m v$ is ghost for some $m$;
the minimal such $m$ is called the \textbf{multiplicity} of the
eigenvalue (and also of the eigenvector).

A tangible vector $v$ is a \textbf{generalized supertropical
eigenvector} of $A$, with \textbf{generalized supertropical
eigenvalue} $\bt \in \tTz$,
 if $$A^m v = \bt^m v + \text{ghost}$$ for some $m$;
the minimal such $m$ is called the \textbf{multiplicity} of the
eigenvalue (and also of the eigenvector). A \textbf{supertropical
eigenvalue} (resp.~\textbf{supertropical eigenvector} is a
 generalized supertropical eigenvalue
(resp.~generalized supertropical eigenvector)  of multiplicity 1.
\end{defn}

(Although weak generalized supertropical  eigenvectors need not be
tangible, generalized supertropical  eigenvectors are required to
be tangible, since we are about to prove that there are ``enough''
of them for a reasonable theory.  Note that if we did not require
$\bt$ to be tangible, all vectors would be weak supertropical
eigenvectors; indeed, for any given matrix $A$ and vector $v$, any
large enough ghost element $\bt $ would be a weak supertropical
eigenvalue of $A$ with respect to $v$. On the other hand, this
observation does not apply to the definition of supertropical
eigenvectors.)

 When $\nu_{\tT}$ is 1:1 (which is the case in the applications to tropical geometry), tangible weak (generalized) supertropical
eigenvectors are (generalized) supertropical eigenvectors, because
of the following observation.

\begin{lem}\label{supereig} Suppose $\nu_{\tT}$ is 1:1. If $v$ is tangible and $A^m v + \bt^m
v$ is ghost for $\bt \in \tT$, then $A^m v = \bt^m v +
\text{ghost}$.
\end{lem}
\begin{proof} Write  $v = (r_1, \dots, r_n)$ where each $r_i \in \tTz$, and $A^m v = (s_1, \dots, s_n).$
 But then $\bt^m r_i \in
\tTz,$ so the $i$-th component $s_i + \bt^m r_i$ of $A^m v + \bt^m
v$ can be ghost only when  $s_i=\bt^m v_i$ or $s_i$ is ghost
dominating $\bt^m v_i$, in which case
$$s_i = s_i + \bt^m r_i = \bt^m v_i + \ghost. $$
\end{proof}

\begin{example}\label{eigen4} The matrix
$A =\vMat{4}{0}{0}{1}$ of ~Example \ref{eigen} also has the
tangible supertropical eigenvector $v = (0, 4)$, corresponding to
the supertropical
 eigenvalue~$1$, since $$Av = (4^\nu, 5) = 1   v +(4^\nu, 0^\nu).$$
\end{example}

\begin{rem}\label{eigen33}
Let $A_\tang $ denote the matrix obtained by replacing each ghost
entry of $A$ by $\fzero.$ Then $A = A_\tang +\ghost,$ so clearly
every (generalized) supertropical eigenvalue of $ A_\tang$ is a
(generalized) supertropical eigenvalue of $A$. This enables us to
reduce many questions about supertropical eigenvalues to tangible
matrices.
\end{rem}

We  also want to study  supertropical eigenvalues in terms of
other notions.

\begin{prop}\label{eigen50} The matrix $A + \beta I$ is
singular, iff $\beta$ is a root of the  characteristic polynomial
$f_A$ of~$A$.
\end{prop}
\begin{proof} The determinant of $A + \beta I$ comes from \nmulti
s of greatest weight. Since the contribution from $\beta I$ comes
from say $n-k$ entries of $\bt$ along the diagonal, the remaining
$k$ entries must come from a $k$-multicycle, in the graph of $A$,
which dominates the $k$-multicycles and has some total weight
$\a_k$. On the other hand, as already noted in the proof of
Theorem \ref{hamilton-Cayley}, $\a_k$ is precisely the coefficient
of $\la^{n-k}$ in $f_A$. Thus, $|A + \beta I| \in \tGz$ iff either
$\a_k \in \tGz$ or some other $\a_{k'} \bt^{k'}$ matches  $\a_{k}
\bt^{k}$ (and dominates all other $\a_{j} \bt^{j}$); but this is
precisely the criterion for $\bt$ to be a root of $f_A,$ proving
the assertion.
\end{proof}

\begin{prop}\label{eigen5} If $v$ is a tangible  supertropical eigenvector of
$A$ with  supertropical eigenvalue $\bt$, the matrix $A + \bt I$
is singular (and thus $\bt$ must be a (tropical) root of the
characteristic polynomial $f_A$ of $A$).
\end{prop}

\begin{proof}
 $(A +\bt I)v$ is ghost, and thus
so is $\adj{{A+\bt I}}(A+\bt I)v.$ If $A+\bt I$ were nonsingular
this would be $f_A(\bt)I_{A+\bt I} \, v$, implying $I_{A+\bt I}$
is ghost, by Lemma \ref{iso2}, a contradiction.
\end{proof}

Our goal is to prove the converse, that every tangible root of the
 characteristic polynomial of $A$ is a supertropical
eigenvalue (of a tangible supertropical eigenvector). First of
all, let us reduce the theory to tangible matrices.

\begin{rem} If $\widehat A$ is a tangible matrix (i.e., all entries
are in $\tTz$), such that $\widehat A^\nu = A^\nu,$ then every
tangible supertropical eigenvector $v$ of $\widehat A$ is a
supertropical eigenvector of $A$, with the same supertropical
eigenvalue. (Indeed, let $\bt$ be the eigenvalue of $v$ for $
\widehat A$. obviously $\widehat A v$ and $A v$ are $\nu$-matched,
with every tangible component of $\widehat A v$ matched by a
tangible component of $Av,$ so $$A v = \widehat Av + \ghost = \bt
v + \ghost.)$$\end{rem}

\begin{thm}\label{eigen8} Assume that $\nu|_\tT : \tT \to \tG$ is 1:1. For any  matrix $A$,
the dominant tangible root of the  characteristic polynomial
${{{f}_A}}^{\essn}=\la^n + \sum_{j=1}^t \a _{k_j} \la ^{n-k_j}$ of
$A$ is an eigenvalue of $A$, and has  a tangible eigenvector. The
matrix $A$ has at least $t$ supertropical tangible eigenvectors,
whose respective tangible eigenvalues are precisely the tangible
roots of ${{{\hat f}_A}}^{\essn}$.
\end{thm}
\begin{proof}  Let $B = A + \bt I$. By Proposition \ref{eigen50},
$B$ is singular, which implies by Corollary~\ref{coldep} that its
columns $c_1, \dots, c_n$ are tropically dependent. Taking
tangible $\gm_1, \dots, \gm_n$, not all of them $\fzero$, such
that $\sum \gm_j c_j \in \tGz^{(n)},$ and letting $v = (\gm_1,
\dots, \gm_n),$ we see that
$$Av + \bt v = B v = \sum \gm_j c_j \in \tGz^{(n)},$$ implying by Lemma
\ref{supereig} that $Av = \bt v + \ghost,$ as desired.
\end{proof}

We have proved that the supertropical eigenvalues  are precisely
the roots of the characteristic polynomial.
 On the other
hand, there may be extra cycles that also contribute weak
supertropical eigenvectors, providing weak supertropical
eigenvalues that  are not roots of the characteristic polynomial.
Let us illustrate this feature.

\begin{example} Let $A$ be the $3 \times 3$ tropical matrix
$$\(  \begin{matrix} -\infty & 14
& 8 \\ 0 & -\infty & -\infty\\ 0  & 1 &   -\infty\end{matrix}\)
,$$ in logarithmic notation.  The tangible characteristic
polynomial is $\la^3 +  14 \la + 9$ whose tangible roots are $7$
and $-5$, and the supertropical tangible eigenvectors
corresponding to $f_A$ are:

\begin{enumerate}
\item $(7, 0, 0)$ of eigenvalue $7,$  which arises from the cycle
$(1,2),\, (2,1)$ of weight $\frac{14}2 = 7$. \pSkip

 \item The tangible supertropical  eigenvector $v=
(0,5,11)$; here
$$Av = (19^\nu, 0, 6) = (-5)v + (19^\nu, -\infty, -\infty) .$$
\end{enumerate}

Note that the other cycles give rise to weak  supertropical
eigenvectors, although not tangible:
\begin{enumerate} \item  The cycle  $(1,3),\,
(3,1)$  yields the  supertropical eigenvector $(10^\nu, 0, 6)$ of
supertropical eigenvalue~$4$. \pSkip
 \item The cycle $(1,3),\, (2,1),\, (3,2)$ of weight $\frac{9}3 = 3$   yields the supertropical eigenvector $(6^\nu, 3^\nu,
0)$ of supertropical eigenvalue $3$.\end{enumerate}
\end{example}

\begin{example} Let $A$ be the $3 \times 3$ tropical matrix $$\(  \begin{matrix} -\infty & -\infty
& 7 \\ 4 & -\infty & -\infty\\ 3 & 5 & -\infty\end{matrix}\)  ,$$
over the extended max-plus semiring $D(\Real)$  (in logarithmic
notation). We look for an eigenvector $(0, \gm_2, \gm_3)$, by
means of rather crude computations.
 For any
supertropical eigenvalue~$\bt,$ we have the three equations (in
$\R ,$ with respect to the familiar addition and multiplication):
\pSkip
\begin{enumerate}
\item $7+\gm_3 = \bt;$ \pSkip \item $4 = \gm_2 + \bt;$ \pSkip \item $\max
\{ 3, 5+ \gm_2\} = \gm_3 + \bt.$\pSkip
\end{enumerate}

Adding the first two equations yields $\gm_2 + \gm_3 + 3 = 0.$
Thus, plugging into (3) yields either $3 = \gm_3 + \bt$ or $3\gm_3
= -5.$ In the former case, we get $v = (0, -1, -2),$ which is not
an eigenvector since $Av = (5,4,4)$! (The reason is that reversing
the steps in the proposed solution does not satisfy (3).)

On the other hand, $v = (0, -1^\nu, -2)$ is a weak supertropical
eigenvector, since $Av = (5,4,4^\nu),$ and then $$Av =  5v +
(0^\nu,0^\nu,4^\nu);$$ thus $5$ is a weak supertropical
eigenvalue. Also $A^2v = (11^\nu, 9,9),$ and  $A^3v = (16, 15^\nu,
14^\nu),$ implying 5 is a supertropical eigenvalue of $A^2v.$  But
these weak supertropical eigenvectors are quite strange, since
$A^3v = 16v+\text{ghost}$,  whereby we see that $v$ is a
generalized supertropical eigenvector for the   generalized
supertropical eigenvalue $\frac {16}3.$

In the latter case, we get $\gm_3 = -\frac 53,$ in which case
$\gm_2 = -\frac 43,$ so $v = (0, -\frac 43, -\frac 53)$, which is
a supertropical eigenvector  with supertropical eigenvalue
$\frac{16}3.$

If one plays a bit more with the equations, one also gets the weak
supertropical eigenvector $(0, -2, -1^\nu)$, with weak
supertropical eigenvalue $6$. But, again,  $A^3v =
16v+\text{ghost}$.

The mystery can be cleared up by examining the characteristic
polynomial $\la ^3 + 10 \la + 16$ of $A$. The essential part of
$f_A$ is $\la ^3 + 16,$ whose only tangible root is $\bt = \frac
{16}3,$ and indeed we get the supertropical eigenvector $(0,
-\frac 43, -\frac 53)$ by applying the proof to $v_0 = (0,
-\infty, -\infty)$ and $\bt = \frac {16}3.$
\end{example}

Here is a surprising counterexample to a natural conjecture.

\begin{example}
Let $A =  \left(\begin{matrix} 0 & 0 \\
1 & 2
\end{matrix}\right),$
of Example \ref{Van3}. Its characteristic polynomial is $\la^2 +
2\la +2 = (\la+0)(\la+2),$ whose roots are $2$ and $0$. The
eigenvalue $2$ has tangible eigenvector $v = (0, 2)$  since $Aw =
(2, 4) = 2v$, but there are no other tangible eigenvalues. $A$
does have the tangible supertropical eigenvalue $0$, with tangible
supertropical eigenvector $w = (2,1),$ since
$$Aw = (2, 3^\nu) = 0w + ( -\infty, 3^\nu).$$
Note that $A + 0I = \left(\begin{matrix} 0^\nu & 0 \\
1 & 2
\end{matrix}\right)$ is singular; i.e., $|A + 0I| = 2^\nu $.

Furthermore, $A^2 =  \left(\begin{matrix} 1 & 2 \\
3 & 4
\end{matrix}\right),$ which is singular, and $$A^4 =   \left(\begin{matrix} 5 & 6 \\
7 & 8
\end{matrix}\right) = 4 A,$$
implying that $A^2 $ is a root of $\la^2 + 4A,$ and thus $A $ is a
root of $g = \la^4 + 4A^2 = (\la(\la+2))^2,$ but $0$ is not a root
of $g$ although it is a root of $f_A$. This shows that the naive
formulation of Frobenius' theorem fails in the supertropical
theory.
\end{example}

 Let us say that a matrix $A$ is \textbf{separable} if its characteristic polynomial $f_A$
  splits as the product of
distinct monic linear tangible factors. (Equivalently, $f_A = \sum
_{i= 0}^n \a _i \la ^i$ is essential, with each $\a_i \in F$
tangible.) Let $U_A$ be the matrix whose columns are supertropical
eigenvectors of $A$. We conjecture  that the matrix $U_A$ is
nonsingular. The argument seems to be rather intricate, involving
a description of the \multi s of $A$ in terms of its eigenvalues,
so, for the time being, we insert this as a hypothesis.

\begin{cor}\label{simdiag}  Every separable $n\times n$ matrix $A$
(for which $U_A$ is nonsingular) is tropically conjugate to a
diagonal matrix, in the sense that
$$ U_A^\nb AU_A = D_A + \ghost,$$  where $D_A$ is the diagonal matrix
whose entries $\{ \bt_1, \dots, \bt_n\}$  are the supertropical
eigenvalues of $A$.
 \end{cor}
 \begin{proof} Suppose $f = \prod_{i=1}^n (\la + \bt _i).$
  Then taking supertropical eigenvectors $v_i$ for which  $$Av_i = \bt _i v_i + \text{ghost},$$
  we have
$AU_A = U_AD_A + \text{ghost},$ implying  $$
\begin{array}{lllll}
  U_A^\nb AU_A  & = &  U_A^\nb U_A D_A + U_A^\nb \text{ghost} & =&    \um'_{U_A} D_A +
\text{ghost} \\[2mm]
   & = &  (\um + \ghost) D_A + \text{ghost} & = &  D_A +
\text{ghost}. \\
\end{array}$$
\end{proof}


\end{document}

%% file: tropMacro.tex
\usepackage{epsfig,latexsym,amsfonts,amssymb,amsmath,amscd,graphics,epic}
\usepackage{amsfonts,amssymb,amsmath,amscd,amsthm}
\usepackage[mathscr]{eucal}
\usepackage{mathrsfs}
\usepackage{oldgerm,units}
\usepackage{wrapfig,epsfig}

\usepackage{ifthen}
\usepackage{mathbbol}

\usepackage{amsthm}
\usepackage[mathscr]{eucal}
\usepackage{mathrsfs}
\usepackage{mathbbol}
\usepackage{oldgerm,units}
\usepackage{wrapfig}

\newtheorem{theorem}{Theorem}[section]
\newtheorem{proposition}[theorem]{Proposition}
\newtheorem{definition}[theorem]{Definition}

\newtheorem{lemma}[theorem]{Lemma}

\newtheorem{example}[theorem]{Example}
\newtheorem{remark}[theorem]{Remark}

\newtheorem{note}[theorem]{Note}

\newcommand{\Ref}[1]{(\ref{#1})}

\newcommand{\Real}{\mathbb R}

\newcommand{\Int}{\mathbb Z}

\newcommand{\one}{\mathbb{1}}
\newcommand{\zero}{\mathbb{0}}

\newcommand{\trop}[1]{\mathcal{#1}}

\newcommand{\tG}{\trop{G}}
\newcommand{\tH}{\trop{H}}

\newcommand{\tJ}{\trop{J}}

\newcommand{\tS}{\trop{S}}
\newcommand{\tT}{\trop{T}}

\newcommand{\tU}{\trop{U}}
\newcommand{\tV}{\trop{V}}
\newcommand{\tW}{\trop{W}}

\newcommand{\To}{\longrightarrow }

\newcommand{\tUniS}{-\infty}

\newcommand{\al}{\alpha}
\newcommand{\bt}{\beta}
\newcommand{\gm}{\gamma}

\newcommand{\lm}{\lambda}

\newcommand{\TrS}{\oplus}

\newcommand{\OP}{\left(}
\newcommand{\CP}{\right)}

    \ifx\proof\undefined
    \newenvironment{proof}{
    \smallskip
    \noindent\emph{Proof.}}{\hfill\(\Box\)
    \bigskip
    } \fi

\newcommand{\vMat}[4]{\OP \begin{array}{cc}
  #1 & #2 \\
  #3 & #4
\end{array}\CP}

\newcommand{\vvMat}[9]{\OP \begin{array}{ccc}
  #1 & #2 & #3\\
  #4 & #5 & #6\\
  #7 & #8 & #9\\
\end{array}\CP}

\newcommand{\ifdef}[3]{\ifthenelse{\equal{#1}{true}}{#2}{#3}}